\documentclass[reqno, 11pt]{amsart}

\usepackage{amssymb}
\usepackage{verbatim}

\DeclareMathAccent{\mathring}{\mathalpha}{operators}{"17}

\newcommand{\mysection}[1]{\section{#1}
\setcounter{equation}{0}}

\newtheorem{theorem}{Theorem}[section]
\newtheorem{corollary}[theorem]{Corollary}
\newtheorem{lemma}[theorem]{Lemma}

\theoremstyle{definition}
\newtheorem{remark}[theorem]{Remark}

\theoremstyle{definition}

\theoremstyle{definition}
\newtheorem{assumption}[theorem]{Assumption}

\makeatletter
\def\dashint{\operatorname%
{\,\,\text{\bf--}\kern-.98em\DOTSI\intop\ilimits@\!\!}}
\makeatother

\newcommand{\nlimsup}%
{\operatornamewithlimits{\overline{lim}}}
\newcommand{\nliminf}%
{\operatornamewithlimits{\underline{lim}}}

\def\bR{\mathbb{R}}
\def\bZ{\mathbb{Z}}

\def\cC{\mathcal{C}}

\def\cH{\mathcal{H}}
\def\cP{\mathcal{P}}

\def\cS{\mathcal{S}}

\newcommand\dist{{\rm dist}\,}

\begin{document}

\title[Existence for fully nonlinear parabolic equations] {On the
existence of smooth solutions for fully nonlinear parabolic equations
with measurable ``coefficients'' without convexity assumptions}

\author[H. Dong]{Hongjie Dong}
\address[H. Dong]{Division of Applied Mathematics, Brown University,
182 George Street, Providence, RI, 02912}
\email{Hongjie\_Dong@brown.edu}
\thanks{H. Dong was partially supported by NSF Grant DMS-1056737.}

\author[N. V. Krylov]{N.V. Krylov}
\address[N. V. Krylov]{127 Vincent Hall, University of Minnesota,
 Minneapolis, MN, 55455}
\email{krylov@math.umn.edu}
\thanks{N. V. Krylov was partially supported by
  NSF Grant DMS-1160569}



\begin{abstract}
We show that for any uniformly parabolic fully
nonlinear second-order equation with bounded measurable
``coefficients'' and bounded ``free'' term
 in  any cylindrical smooth domain  with  smooth
boundary data
one can find an
approximating equation which has a unique continuous  solution
with the first derivatives bounded and the second spacial derivatives
locally bounded. The approximating equation is constructed in such a
way that it modifies the original one only for large values of the
unknown function and its spacial derivatives.
\end{abstract}

\maketitle

\mysection{Introduction and  main result}

                                             \label{section 2.5.1}

In this article, we consider parabolic  equations \begin{equation}
                                                \label{7.29.1}
\partial_t v(t,x)+H[v](t,x):= \partial_t v(t,x)+H(v(t, x),D v(t,x),D^{2}v(t, x),t, x)=0
\end{equation}
in subdomains of $\bR^{d+1} =\bR\times \bR^d $, where
$$
\bR^{d}=\{x=(x_{1},...,x_{d}):x_{1},...,x_{d}\in \bR\}.
$$
Here
$$
\partial_t=\partial/\partial t,\quad
 D^{2}u=(D_{ij}u),\quad Du=(D_{i}u),\quad
D_{i}=\frac{\partial}{\partial x_{i}},
\quad D_{ij}=D_{i}D_{j}.
$$
 We prove that for any uniformly parabolic fully nonlinear second-order equation with bounded measurable ``coefficients'' and bounded ``free'' term in a given cylindrical smooth domain  with  smooth
boundary data, one can find an approximating equation which has a
unique continuous solution with
the first derivatives bounded and the second spacial derivatives locally bounded.
The novelty of our result is that we do not impose any convexity
assumptions on the equation. This is a continuation of
\cite{Kr12.2}, in which a similar result was obtained for elliptic
equations.

The convexity of operators plays an important role in the regularity
theory of fully nonlinear elliptic and parabolic equations. For
elliptic equations without convexity assumptions, the best result
one can get is that viscosity solutions are in $C^{1+\alpha}$ (see
Trudinger \cite{Tr89}) under the condition that the operators are
sufficient regular (H\"older) with respect to the independent
variables. In fact, N. Nadirashvili and S. Vl\v{a}dut \cite{NV}
found an example which shows that even for elliptic operators
independent of the space variables viscosity solutions may not have
bounded second-order derivatives. For equations with measurable
coefficients, M. G. Crandall, M. Kocan, and A. \'Swi{\c e}ch
\cite{CKS00} developed a theory of $L_{p}$-viscosity solutions  (see
also the references therein).

Interior $W^{2}_{p}$ a priori estimates for elliptic equations was first derived by L.
Caffarelli under an assumption that certain estimates hold for equations
with zero ``free'' term,
  which are known to hold only for $H$ that are
either convex or concave with respect to $v$, $Dv$, and $D^{2}v$ (see
\cite{Caf89} and \cite{CC95}). Note that some particular cases of $C^{2+\alpha}$ a priori estimates without this assumption can be found in \cite{CC03} and \cite{Ko09}. This line of research
was continued by L. Wang in \cite{Wa92} who obtained similar interior a
priori estimates for parabolic equations, by M. G. Crandall, M. Kocan,
and A. \'Swi{\c e}ch \cite{CKS00} who established the {\em
solvability\/} in local Sobolev spaces of the boundary-value problems
for fully nonlinear parabolic equations, and by  N.~Winter \cite{Wi09}
who established the solvability in the global $W^2_p$-space of the
associated boundary-value problem in the elliptic case. In the existence
parts in \cite{CKS00} and \cite{Wi09} the function
 $H$ is supposed to be convex with respect to
$ D^{2}v$ and continuous in $x$ (concerning the latter assumption see
\cite[Remark 2.3]{Wi09}, \cite{Kr10}, and
 \cite[Example 8.3]{CKS00}). It is worth noting that in
the above references the authors considered  equations like \eqref{7.29.1} with the right-hand side which is not zero but rather a function from an $L_{p}$-space. In our setting we can only treat bounded right-hand
sides.

In two recent papers \cite{Kr10, DKL} the authors used a very
different approach to study the $W^2_p$ theory of fully nonlinear
elliptic and parabolic equations with VMO ``coefficients''. The
convexity of $H$ with respect to $D^{2}v$ is relaxed for the
a priori estimates, but is still assumed in the proof of the existence
result. Nevertheless, it is conjectured in \cite{DKL} that the
convexity condition can be dropped or at least relaxed for the
existence result.

This conjecture was addressed in \cite{Kr12.2} and \cite{Kr12.3}. In \cite{Kr12.2} the author considered uniformly elliptic fully nonlinear second-order equation of the form $H[v]=0$ with bounded measurable ``coefficients'' and
bounded ``free'' term  in a given smooth domain  with  smooth
boundary data. It is shown that one can find an approximating equation
$$
\max(H[v],P[v]-K)=0,
$$
which has a
unique continuous   solution with locally bounded second-order derivatives. The approximating equation differs from the original one only for large values of the unknown
function and its derivatives. By using this result, in \cite{Kr12.3} the author established the existence and uniqueness of solutions of fully nonlinear elliptic second-order equations in smooth domains, under a relaxed convexity assumption
with respect to $D^2 v$ and a VMO condition with respect to $x$ which are imposed only for large $|D^2 v|$.

Roughly speaking, the main idea of \cite{Kr12.2} is that on the set, say $\Gamma$, where the second-order
derivatives of $v$ are large we have $P[v]=K$ and  in the spirit of the maximum principle the second order derivative on $\Gamma$ are controlled by
their values on the boundary of $\Gamma$, where they are under control by
the definition of $\Gamma$. The implementation of this idea, however,
requires sufficient regularity of solutions to \eqref{9.23.2}. Since this
is not known a priori, the above idea is applied at the level of
finite differences.

In this article, we extend the result of \cite{Kr12.2} to
parabolic equations. To state our main results, we introduce
a few notation and assumptions.
 Let  $\cS$  be  the
set of symmetric $d\times d$ matrices, fix a constant
$\delta\in(0,1]$, and  set
$$
\cS_{\delta}=\{a\in\cS:\delta|\xi|^{2}\leq a_{ij}\xi_{i}\xi_{j}
\leq\delta^{-1}|\xi|^{2},\quad \forall\, \xi\in\bR^{d}\}, $$ where and
everywhere in the article the summation convention is enforced unless
specifically stated otherwise.

\begin{assumption}
                                    \label{assumption 9.23.1}
(i) The function $H(u,t,x)$, $u=(u',u'')$, $$
u'=(u'_{0},u'_{1},...,u'_{d}) \in\bR^{d+1},\quad u''\in\cS,\quad
(t,x)\in\bR^{d+1}, $$
 is measurable with respect to $(t,x)$ for any $u$,
and Lipschitz continuous in $u$ for every $(t,x)\in\bR^{d+1}$.

(ii) For any $(t,x)$, at all points of differentiability of $H(u,t,x)$ with
respect to $u$, we have
$$
(H_{u''_{ij}})\in \cS_{\delta},\quad
|H_{u'_{k}}|\leq \delta^{-1},\quad k=1,...,d, \quad 0\leq-H_{u'_{0}}\leq
\delta^{-1}.
$$

(iii) Finally,  $$ \bar{H}:= \sup_{(t,x) \in\bR^{d+1}}|H(0,t,x)|<\infty. $$
\end{assumption}
\begin{remark}
                                    \label{remark 12.21.1}
It is almost obvious that Assumption \ref{assumption 9.23.1} (ii) is
equivalent to the requirement that, for any $u\in\bR^{d+1}\times\cS$,
$x,\xi\in\bR^{d}$, $\eta\in\{\pm e_{1},...,\pm e_{d}\}$, where
$e_{1},...,e_{d}$
 is the set
of standard basis vectors in $\bR^{d}$, and $r\geq0$, we have
$$
\delta|\xi|^{2}\leq H(u',u''+\xi\xi^{*},t,x)- H(u',u'',t,x)\leq
\delta^{-1}|\xi|^{2}, $$ $$ |H(u'+r(0,\eta),u'',t ,x)-H(u',u'',t ,x)|\leq
\delta^{-1}r,
$$
$$
 H(u',u'',t ,x)-\delta^{-1}r
\leq H(u'+r(1,0),u'',t ,x)\leq H(u',u'',t ,x) ,
$$
where
$(0,\eta)=(0,\eta_{1},...,\eta_{d})$ and $(1,0)=(1,0,...,0)$.
\end{remark}

Let $\Omega$ be an open bounded subset of $\bR^{d}$ with
$C^{2}$ boundary and $-\infty
\le S<T<\infty$. We denote the parabolic boundary of the cylinder
$(S,T)\times \Omega$ by
$$
\partial' ((S,T)\times \Omega)=(\{T\}\times \Omega)\cup ((S,T]\times \partial \Omega).
$$
For any $T>0 $, we define
$\Omega_T=(0,T)\times \Omega$.

We use the H\"older spaces $\cC^{\alpha,\beta},\alpha,\beta\in
(0,1]$, of functions of  $(t,x)$ which are the
spaces of bounded functions having finite H\"older
  norm  of
order $\alpha$ in
$t$ and $\beta$ in $x$.
 The symbol $C^{1,2}$ stands for the space
of bounded
 functions $u$ for which $\partial_{t}u, Du$, and $D^{2}u$
are bounded and continuous with respect to $(t,x)$.
These spaces are provided with natural norms.
We denote by $W^{1,2}_p(\Omega_T)$  the space of
functions $v$ defined on
$\Omega_T$ such that $v$, $Dv$, $D^2v $, and $\partial_t v$
are in
$L_p(\Omega_T)$.

\begin{theorem}
                                    \label{theorem 9.23.1}
Let $T>0$ and $K\geq0$ be fixed constants, and
$g\in  C ^{1,2} (\bar\Omega_T)$.
There is a  constant  $\hat{\delta}\in(0,\delta]$ depending only on
$\delta$ and $d$ and there exists a function $P(u) $ (independent of
$t,x$), satisfying Assumption \ref{assumption 9.23.1} with $\hat{\delta}$
in  place of $\delta$, such that the equation
\begin{equation}
                                               \label{9.23.2}
\partial_t v+\max(H[v],P[v]-K)=0
\end{equation}
in $\Omega_T$ (a.e.) with terminal-boundary
condition $v=g$ on $\partial'\Omega$ has a unique solution $v\in
\cC^{1,1}(\bar{\Omega}_T)\cap W^{1,2}_{\infty,\text{loc}}(\Omega_{T})$.
In addition, for all
$i,j$, and $p\in(d+1,\infty)$,
\begin{equation}
                                                \label{1.13.1}
|v|,|D_{i}v|,\rho|D_{ij} v |,|\partial_t v |\leq N(\bar{H}+K
+\|g\|_{ C ^{1,2}(\Omega_T)})\quad\text{in} \quad \Omega_T \quad (a.e.),
\end{equation}
\begin{equation}
                                                \label{1.13.2}
\|v\|_{W^{1,2}_{p}(\Omega_T)}\leq
N_{p}(\bar{H}+K+\|g\|_{W^{1,2}_{p}(\Omega_T)}), \end{equation}
\begin{equation}
                                                \label{2.28.1}
\|v\|_{\cC^{\alpha/2,\alpha}(\Omega_T)}\leq N (\|H[0]\|_{L_{d+1}(\Omega_T)}+\|g
\|_{\cC^{\alpha/2,\alpha}(\Omega_T)}),
\end{equation}
where
$$
\rho=\rho(x)=\dist(x,\bR^{d}\setminus\Omega),
$$
$\alpha\in(0,1)$ is a constant depending only on $d$ and $\delta$,
   $N$ is a constant depending only on $\Omega$ and $\delta$, whereas $N_{p}$ only depends on $\Omega$,
$T$, $\delta$, and
$p$.

Finally, $P(u )$ is constructed on the sole basis of $\delta$ and $d$,
it is  positive homogeneous of degree one and convex in $u$.
\end{theorem}

 In the proof of Theorem \ref{theorem 9.23.1},
we adapt the aforementioned idea in \cite{Kr12.2} to the
parabolic setting. As there, we start at the level of
finite  differences. Although it is tempting to discretize
the equation with respect to both $t$ and $x$, it turns out
that it suffices for us to discretize only with respect to $x$,
so that the discretized equation is a system of ordinary
differential equations with respect to $t$. The estimates of
the solution to the discretized equation as well as its
first-order space finite differences follow the line in
\cite{Kr12.2} by using a version of the maximum principle in
``non-cylindrical'' domains; cf. Lemma \ref{lemma 7.21.1}.
The estimates of the second-order space finite differences
are more involved. In order to get their lower bound, we
apply Bernstein's method to the discretized equation. In
contrast to the elliptic case, for the upper bound we first
need to control the time derivative of the solution, using
again Lemma \ref{lemma 7.21.1}. The upper bound of the
second-order space finite differences is then deduced from
the above estimates and the equation itself.

\begin{remark}
                                    \label{remark 2.28.1}
Estimate \eqref{2.28.1} follows from other assertions of Theorem
\ref{theorem 9.23.1} and the classical results about linear equations
with measurable coefficients (see, for instance, Section VII.9 of
\cite{GML}). Indeed, as is easy to see for $v\in
 W^{1,2}_{p}(\Omega_T) $
satisfying \eqref{9.23.2} we have that $$
-\max(H[0],P[0]-K)=\max(H[v],P[v]-K)-\max(H[0],P[0]-K) $$ $$
=a_{ij}D_{ij}v+b_{i}D_{i}v-cv $$ with some functions
$a=(a_{ij})\in\cS_{\hat{\delta}}$, $|b_{i}|\leq\hat{\delta}^{-1}$,
$0\leq c\leq\hat{\delta}^{-1}$  (cf. the proof of Lemma \ref{lemma
10.4.2}). Furthermore,
$$
|\max(H[0],P[0]-K)| =|\max(H[0],-K)| \leq|H[0]|.
$$

The assertion of Theorem \ref{theorem 9.23.1}
concerning uniqueness in our class of functions is also a classical
result derived from the parabolic Alexandrov estimate. \end{remark}

\begin{remark}
                                 \label{remark 8.14.1}

Even though quite a few auxiliary  results from  \cite{Kr12.2}
are used in the present article, the main result of
\cite{Kr12.2} is not. It even turns out that it  can be
derived from Theorem \ref{theorem 9.23.1}  and
the results of \cite{DKL}. Of course, such an
indirect derivation is somewhat longer than the one given
in \cite{Kr12.2} but yet it is worth mentioning.

Thus, we assume that $H$ and $g$ are independent of $t$.
The proof  of the elliptic counterparts of \eqref{1.13.2}
and \eqref{2.28.1} consists of just a repetition of the arguments
of the present article (using \cite{DKL}).
 In what concerns
existence and estimate \eqref{1.13.1}, we
 denote by $v_{T}$ the solution from Theorem
\ref{theorem 9.23.1}. By \eqref{1.13.1}, for any
$S\geq0$ the family
$v_{T}$, $T\geq S$, is equi-bounded and equi-continuous
on $\Omega_{S}$. It follows that there is a sequence
$T(n)\to\infty$ as $n\to\infty$ such that $
v_{T(n)}$ converge uniformly on each $\Omega_{S}$
to a function $v$ obviously satisfying
\eqref{1.13.1} on $\Omega_{\infty}$.
The rules of passing to the limit in
fully nonlinear equations (see, for instance, Theorem 3.5.9 of
\cite{Kr85})
show  that $v$ satisfies \eqref{9.23.2}
in $\Omega_{\infty}$. Since the functions $g$, $H$, and $P$
are independent of $t$, $v(t+T,x)$ satisfies
the same equation for any fixed
$T\geq0$ and by uniqueness $v(t,x)=v(t+T,x)$.
This means that $v(t,x)=v(x)$, equation
\eqref{9.23.2} becomes elliptic, and we obtain
all assertions of Theorem 1.1 of \cite{Kr12.2}.
\end{remark}

To conclude our comments about Theorem \ref{theorem 9.23.1}
 we show how $P$
is constructed. By Theorems 3.1 of \cite{Kr11} there exists a set $$ \{
l_{1},...,l_{m}\} \subset \bZ^{d }, $$ $m=m(\delta,d)\geq d$, chosen on
the sole
 basis of
knowing $\delta$ and $d$ and
 there exists
a constant  $$ \hat{\delta}=\hat{\delta}(\delta,d ) \in(0,\delta/4] $$
such that:

(i) We have $$ e_{i},e_{i}\pm e_{j}\in \{l_{1},...,l_{m}\}
=\{-l_{1},...,-l_{m}\} $$ for all $i,j=1,...,d$  (recall that
$e_{1},...,e_{d}$ is the
 standard orthonormal basis of $\bR^{d}$);

(ii) There   exist real-analytic functions $\lambda_{1}(a),...,
\lambda_{m}(a)$ on $\cS_{\delta/4}$ such that for any
$a\in\cS_{\delta/4}$
\begin{equation*}
a\equiv\sum_{k=1}^{m}\lambda_{k}(a)l_{k}l_{k}^{*}, \quad
\hat{\delta}^{-1}\geq\lambda_{k}(a)\geq\hat{\delta} ,\quad \forall\, k.
\end{equation*}

Now introduce
$$ \cP(z)  =\max_{\substack{\hat{\delta}/2\leq a_{k}\leq
2\hat{\delta}^{-1} \\k=1,...,m} } \max_{\substack{ |b_{k}|\leq
2\hat{\delta}^{-1}\\k=1,...,d} } \max_{\hat{\delta}/2 \leq c\leq
2\hat{\delta}^{-1}}\big[\sum_{k=1}^{m} a_{k}
z''_{k}+\sum_{k=1}^{d}b_{k}z'_{k} -cz'_{0}\big], $$ and for
$u=(u',u'')\in\bR^{d+1}\times\cS$ define $$ P(u',u'')=\cP(u',\langle
u''l_{1},l_{1}\rangle,..., \langle u''l_{m},l_{m}\rangle), $$ where
$\langle\cdot,\cdot\rangle$ is the scalar product in $\bR^{d}$.

 The remaining part  the article is organized as follows.
Sections \ref{section 12.13.1} and \ref{section 12.13.2} are
devoted to the reduction of  proving
Theorem \ref{theorem 9.23.1} to proving
Theorem \ref{theorem 10.5.1}, that is a special case of
Theorem \ref{theorem 9.23.1} but under additional
assumptions.
  In Section \ref{section 9.22.1} we consider
finite-difference approximations for equations   with
``constant'' coefficients and prove interior estimates for
the second-order differences of solutions. In Section
\ref{section 10.18.2} by using the results of the previous
section we prove an analog of Theorem \ref{theorem 9.23.1}
for $H$, that, as far as the dependence
on $D^{2}v$ is concerned, include only {\em pure\/} second-order derivatives. We
complete the proof of Theorem \ref{theorem 10.5.1} in
Section \ref{section 12.13.4}.

\mysection{Reducing Theorem \protect\ref{theorem 9.23.1} to a particular
case where $-H_{u'_{0}}\geq \delta$}
                                   \label{section 12.13.1}

Suppose that Theorem \ref{theorem 9.23.1} is true under the additional
assumption that \begin{equation}
                                                   \label{3.5.2}
-H_{u'_{0}}\geq \delta \end{equation}
 at all points of
differentiability of $H(u,t,x)$ with respect to $u$.
 Then we are going to prove it
in the original form. Take an $H$ satisfying only Assumption
\ref{assumption 9.23.1}, take $n>0$, and  consider the mapping
$T_{n}:w\to v$ defined for any $w\in C(\bar{\Omega}_T)$ and mapping it
into a unique solution of \begin{equation}
                                          \label{10.4.1}
\partial_t u+\max(H[v]- v+n\chi(w/n),P[v]-K)=0 \end{equation} in $\Omega$  (a.e.)
with terminal-boundary condition $v=g$ on $\partial'\Omega_T$, where
$$ \chi(t)=(-1)\vee t\wedge 1. $$
By assumption $v$ is well defined and $v=T_{n}w\in
\cC^{1,1}(\bar{\Omega}_T)
\cap W^{1,2}_{\infty,\text{loc}}(\Omega_{T})$  and satisfies
$$
|v|,|D_{i}v|,\rho|D_{ij} v |,|\partial_{t} v |\leq N(\bar{H}+n+K
+\|g\|_{C^{1,2}(\Omega_T)}),
$$
(a.e.) in $\Omega_T$, and
$$
\|v\|_{W^{1,2}_{p}(\Omega_T)}\leq
N_{p}(\bar{H}+n+K+\|g\|_{W^{1,2}_{p}(\Omega_T)})
$$
if $p>d+1$. It follows that, for
each $n$, $T_{n}$ maps $C(\bar{\Omega}_T)$ into its compact subset.

\begin{lemma}
                                        \label{lemma 10.4.1}
For each $n$, the mapping $T_{n}$ is continuous in $C(\bar{\Omega}_T)$.
\end{lemma}

\begin{proof}
Let $w,w_{m}\in C(\bar{\Omega}_T)$, $m=1,2,...$, and assume that
$\|w-w_{m}\|_{0,\Omega_T}\to0$ as $m\to\infty$, where
$\|\cdot\|_{0,\Omega_T}$ is the sup norm in $C(\bar{\Omega}_T)$. In light of
uniqueness of solutions of \eqref{10.4.1} with terminal-boundary condition
$v=g$, to prove the lemma, it suffices to show that, at least along a
subsequence, $\| T_n w -v_{m}\|_{0,\Omega_T}\to 0$,
where
$v_{m}=T_{n}w_{m}$. Since $T_{n} C(\bar{\Omega}_T)$ is a compact set,
there is a subsequence and a $v \in C(\bar{\Omega}_T)$ such that
$\|v-v_{m}\|_{0,\Omega_T}\to 0$ and $v=g$ on $\partial'\Omega_T$. Without
losing generality we may assume that the above convergence holds along
the original sequence. Now we need only  show  that $v=T_{n}w$.

Observe that for $m\geq r$ we have
$$
\partial_t v_m+\max(H[v_{m}]- v_{m}+n\sup_{k\geq
r} \chi(w_{k}/n),P[v_{m}]-K) \geq0
$$ in $\Omega_T$ (a.e.). Since the
norms $\|v_{m}\|_{W^{1,2}_{d+1}(\Omega
_{T})}$ are bounded, by Theorem 3.5.9 of \cite{Kr85}, whose conditions are easily checked on the basis
of Remark \ref{remark 12.21.1},  we have  (a.e.)
$$
\partial_t v+\max(H[v ]-
v+n\sup_{k\geq r}\chi(w_{k}/n),P[v ]-K) \geq0.
$$
By letting
$r\to\infty$ we get  (a.e.)
$$
\partial_t v+ \max(H[v ]- v+n \chi(w /n),P[v ]-K)
\geq0.
$$
One obtains the opposite inequality  starting with
$$
\partial_t v_m+\max(H[v_{m}]- v_{m}+n\inf_{k\geq r}\chi(w_{k}/n),P[v_{m}]-K) \leq0.
$$
It follows that $v=T_{n}w$ indeed and the lemma is proved.
\end{proof}

Now by Tikhonov's theorem we conclude that, for each $n$, there exists
$v^{n}\in C(\bar{\Omega}_T)$ such that $v^{n}= T_{n}v^{n}$. By assumption
$v^{n} \in \cC^{1,1}(\bar{\Omega}_T)\cap
W^{1,2}_{\infty,\text{loc}}(\Omega_T)$  and
\begin{equation}
                                           \label{10.4.03}
|v^{n}|,|D_{i}v^{n}|,\rho|D_{ij} v^{n} |,|\partial_t v^{n} |\leq N(\bar{H}+ \|v^{n}\|_{0,\Omega_T}
+K+\|g\|_{C^{1,2}(\Omega_T)})
\end{equation}
(a.e.) in $\Omega_{T}$ and
\begin{equation}
                                           \label{10.4.3}
\|v^{n}\|_{W^{1,2}_{p}(\Omega_T)}\leq N_{p}(\bar{H}+\|v^{n}\|_{0,\Omega_T}
+K+\|g\|_{W^{1,2}_{p}(\Omega_T)}),
\end{equation}
 where $N$ only depends on $\Omega$ and $\delta$,
 and $N_{p}$ only depends on $\Omega$, $T$, $\delta$, and $p$.

\begin{lemma}
                                        \label{lemma 10.4.2}
There is a constant $N$
depending only on the diameter of $\Omega$ and
$\delta$ such that $$ \|v^{n}\|_{C(\Omega_{T})}\leq
N(\bar{H}+K+\|g\|_{C(\Omega_{T})}).
$$

\end{lemma}

\begin{proof}
Introduce $$ H^{n}_{K}(u,
t,x)=\max(H(u,t,
x)-u'_{0}+n\chi(u'_{0}/n),
P(u)-K) $$ and observe that
  $ H^{n}_{Ku'_{0}}\leq 0$ and
by  Hadamard's formula
$$
H^{n}_{K} (u' ,u'',t,
x)-H^{n}_{K} (0,t,x) =
u''_{ij} \int_{0}^{1}H^{n} _{Ku''_{ij}} (
\theta u' ,
\theta u'',t,x)\,d\theta
$$
\begin{equation}
                                                           \label{7.25.1}
+\sum_{i\geq1}
 u' _{i} \int_{0}^{1}H^{n} _{Ku' _{i}}
(\theta u' ,
\theta u'',t,x)\,d\theta+
 u' _{0} \int_{0}^{1}H^{n} _{Ku' _{0}}
(\theta u',\theta u'',t,x)\,d\theta.
\end{equation}
 provided that $H^{n} (u,t,x)$
is differentiable with respect to $u$
at $(\theta u,x)$ for almost all $\theta\in[0,1]$. Since
this happens to be the case for almost all
$u$, we see that, for each $n$, there exist $\cS_{\delta}$-valued
function $a $ and real-valued
 functions $b_{1},...,b_{d}$, $c$, and $f$
satisfying $|b_{i}|\leq \delta^{-1}$, $ c\geq0$, $|f|\leq \bar{H}+K$
such that in $\Omega$ (a.e.)
$$
\partial_t v^n+a_{ij}D_{ij}v^{n}+b_{i}D_{i}v^{n}-cv^{n}=f.
$$
Now our result follows by the parabolic Alexandrov maximum
principle (see, for instance, Section 3.3
of \cite{Kr85})
 and using the global barrier function
 given, for instance,
in  the proof
of Lemma 8.8 of \cite{Kr11}.  The lemma is proved.
\end{proof}

 Due to this lemma one can drop $\|v^{n}\|_{0,\Omega}$
 on  the right-hand sides of estimates \eqref{10.4.03}
and \eqref{10.4.3}. After that it only
remains to observe that for $n\geq\|v^{n}\|_{0,\Omega}$, the function
$v^{n}$ satisfies \eqref{9.23.2} since $\chi(v^{n}/n)=v^{n}/n$ and
Theorem \ref{theorem 9.23.1} holds in its original form.

Hence, in the rest of the article we suppose that \eqref{3.5.2} holds at
all points of differentiability of $H$ with respect to $u$.

\mysection{Further reductions of Theorem \protect\ref{theorem 9.23.1}}

                                   \label{section 12.13.2}

\noindent{\bf 1}. First, we show that we may additionally assume that
for any
$s,t\in \bR$, $x,y\in\bR^{d}$ and $u=(u',u'')$
\begin{equation}
                                                  \label{9.31.2}
|H(u,t,x)-H(u,s,y)|\leq N
(|t-s|+|x-y|)(1+|u|),
\end{equation}
where $N$ is
independent of $t,s,x,y,u$.

Indeed, if Theorem \ref{theorem 9.23.1} is true in this particular case,
take a nonnegative $\zeta\in C^{\infty}_{0}(\bR^{d+1})$, which integrates
to one,  set $\zeta^{n}(x)=n^{d+1}\zeta(nt,nx)$,
and introduce $ H^{n}(u,t,x)$
as the convolution of $H(u,t,x)$ and $\zeta^{n}$ performed with respect to
$(t,x)$. Observe that $H^{n}$  satisfies \eqref{3.5.2} and Assumption
\ref{assumption 9.23.1} with the same constant $\delta $, whereas
$$
|H^{n}(u,t,x)-H^{n}(u,s,y)|\leq  n
(|t-s|+ |x-y|)
\sup_{r,z}|H(u,r,z)| \|\zeta\|_{\cC^{1}(\bR^{d+1})}
$$
and \eqref{9.31.2} is satisfied since
$$
|H(u,r,z)|\leq |H(0,r,z)|+N(d)\delta^{-1}|u|.
$$ Then assuming that the assertions of
Theorem \ref{theorem 9.23.1} are true under our additional assumption,
we conclude that there exist solutions $v^{n} \in
\cC^{1,1}(\bar{\Omega}_{T})
\cap W^{1,2}_{\infty,\text{loc}} (\Omega_T) $ of
\begin{equation*}
\partial_t v^n+\max(H^{n}[v^{n}],P[v^{n}]-K)=0
\end{equation*}
in $\Omega_T$ (a.e.) with
terminal-boundary condition $v^{n}=g$, for which estimates \eqref{1.13.1} and
\eqref{1.13.2} hold with $v^{n}$ in place of $v$ with the constants $N$
and $N_{p}$ from Theorem \ref{theorem 9.23.1} and with
$$
\overline{H^{n}}=\sup_{(t,x)\in\bR^{d+1}}|H^{n}(0,t,x)| \quad\quad(\leq\bar{H})
$$ in place of $\bar{H}$. In particular,
\begin{equation}
                                              \label{10.3.4}
\partial_t v^m+\check{H}^{n}_{K}[v^{m}]\geq0
\end{equation} in $\Omega_T$ (a.e.) for all
$m\geq n$, where $$ \check{H}^{n}_{K}(u,t,x):=\sup_{k\geq n}
\max(H^{k}(u,t,x),P(u)-K). $$

 Furthermore, being uniformly bounded and uniformly continuous,
the sequence $\{v^{n}\}$ has a subsequence uniformly converging to a
function $v$, for which \eqref{1.13.1} and \eqref{1.13.2}, of course,
hold and $v  \in \cC^{1,1}(\bar{\Omega}_{T})
\cap W^{1,2}_{\infty,\text{loc}} (\Omega_T)   $. In light of \eqref{10.3.4} and the fact that
 the norms $\|v^{n}\|_{W^{1,2}_{p}(\Omega_T)}$ are bounded,
by Theorem 3.5.9 of \cite{Kr85} (the applicability of which
is shown by an argument similar to the one in Remark \ref{remark 2.28.1}) we have
\begin{equation*}
\partial_t v+\check{H}^{n}_{K}[v ]\geq0
\end{equation*}
in $\Omega_{T}$ (a.e.).

Then we notice that by the Lebesgue differentiation theorem
for any $u$
\begin{equation}
                                              \label{10.3.1}
\lim_{n\to\infty}\check{H}^{n}_{K}(u,
t,x)=\max(H(u,t,x),P(u)-K)
\end{equation}
for almost all $(t,x)$. Since $\check{H}^{n}_{K}(u,t,x)$ are
Lipschitz continuous in $u$ with a constant independent of
 $t,x$, and $n$,
there exists a subset of $\Omega_{T}$ of full measure such that
\eqref{10.3.1} holds on this subset for all $u$.

We conclude that in $\Omega_T$ (a.e.)
\begin{equation*}
\partial_t v+\max(H [v],P[v]-K)\geq0.
\end{equation*}
The opposite inequality is obtained by considering $$ \inf_{k\geq n}
\max(H^{k}(u,t,x),P(u)-K). $$

\noindent{\bf  2}. Next,
 we show that one may assume that $H$ is
boundedly inhomogeneous with respect to $u$. Introduce $$
P_{0}(u)=\max_{a\in\cS_{\delta/2}}
\max_{\substack{|b_{i}|\leq2\delta^{-1}\\
i=1,...,d}}\max_{c\in[\delta/2,2
\delta^{-1}]}(a_{ij}u''_{ij}+b_{i}u'_{i}-cu'_{0}), $$ where  the
summations   are performed
 before the maximum is taken.
It is easy to see that $P_{0}[u]$ is a kind of Pucci's operator:
$$
P_{0}(u)=-(\delta/2)\sum_{k=1}^{d}\lambda_{k}^{-}(u'')
+2\delta^{-1}\sum_{k=1}^{d}\lambda_{k}^{+}(u'') $$ $$
+2\delta^{-1}\sum_{k=1}^{d}|u'_{k}|-(\delta/2)(u'_{0})^{+}
+2\delta^{-1}(u'_{0})^{-},
$$
where
$\lambda_{1}(u''),...,\lambda_{d}(u'')$ are the eigenvalues of $u''$ and
$a^{\pm}=(1/2)(|a|\pm a)$.

Recall that the function $P$ is introduced in the end of Section
\ref{section 2.5.1} and observe that
$$
 P(u)  =\max_{\substack{\hat{\delta}/2\leq a_{k}\leq
2\hat{\delta}^{-1} \\k=1,...,m} } \max_{\substack{ |b_{i}|\leq
2\hat{\delta}^{-1}\\i=1,...,d } } \max_{\hat{\delta}/ 2\leq c\leq2
\hat{\delta}^{-1}}\big[\sum_{i,j=1}^{d} \sum_{k=1}^{m}
a_{k}l_{ki}l_{kj}u''_{ij}
 +\sum_{i=1}^{d}b_{i}u'_{i} -cu'_{0}\big].
$$
Moreover,
 owing to property (ii) in the end of Section
\ref{section 2.5.1}, the collection of matrices
$$
\sum_{k=1}^{m}a_{k}l_{k}l_{k}^{*}
$$ such that $\hat{\delta}\leq
a_{k}\leq \hat{\delta}^{-1},k=1,...,m$, covers   $\cS_{\delta/4}$. By
combining this with  the fact that $\hat{\delta} \leq\delta/2 $
(actually, $\hat{\delta} \leq\delta/4$, which will be used much later)
we see that
$$
P(u)\geq -(\delta/4)\sum_{k=1}^{d}\lambda_{k}^{-}(u'')
+4\delta^{-1}\sum_{k=1}^{d}\lambda_{k}^{+}(u'')
$$
$$
+4\delta^{-1}\sum_{k=1}^{d}|u'_{k}|-(\delta/4)(u'_{0})^{+}
+4\delta^{-1}(u'_{0})^{-}
$$
\begin{equation}
                                          \label{3.6.1}
\geq P_{0}(u)+(\delta/4)\sum_{k=1}^{d}|\lambda_{k} (u'')|
+(\delta/4)\sum_{k=0}^{d}|u'_{k}|.
\end{equation}

In particular, $P_{0}\leq P$ and therefore,
$$
\max(H,P-K)=\max(H_{K},P-K),
$$
where $H_{K}=\max(H,P_{0}-K)$. It is
easy to see that the function $H_{K}$ satisfies Assumption
 \ref{assumption 9.23.1} and  \eqref{3.5.2}
 with $\delta/2$  in place of $\delta$.
It also satisfies \eqref{9.31.2} with the same constant $N$.

Furthermore, we have the following.
\begin{lemma}
                                          \label{lemma 9.29.2}
There is a constant $\kappa>0$ depending only on $\delta$ and $d$ such
that  for all $(t,x)\in\Omega_T$ and $u=(u',u'')$
\begin{equation}
                                            \label{9.29.2}
H \leq  P_{0}-\kappa
 \big(\sum_{i,j}|u''_{ij}| +\sum_{i}|u'_{i}|\big)+
 H(0,t,x),
\end{equation}
 \begin{equation}
                                            \label{3.6.3}
H_{K} \leq  P -\kappa
 \big(\sum_{i,j}|u''_{ij}| +\sum_{i}|u'_{i}|\big)+
\ H^{+}(0,t,x).
\end{equation}
Furthermore,
$$
 H(u, t, x) \leq N\big(\sum_{i,j}|u''_{ij}| +\sum_{i}|u'_{i}|\big)+
 H(0, t, x),
$$
$$
|H(u,t,x)|\leq N\big(\sum_{i,j}|u''_{ij}| +\sum_{i}|u'_{i}|\big)+
|H(0,t,x)|,
$$
where the constant $N$ depends only on  $\delta$.
\end{lemma}

\begin{proof}
Observe that if a number $p\in( a,b)$, $a<b$, and $y\in\bR$,
then $$ yp\leq y^{+}b-y^{-}a. $$ Then from  Hadamard's formula
$$
H (u',u'',t,x)-H (0,0,t,x) =u''_{ij}\int_{0}^{1}H _{u''_{ij}} (su' ,su'',t,x)\,ds
$$ $$ +\sum_{i\geq1} u' _{i}\int_{0}^{1}H _{u' _{i}} (su' ,su'',t,x)\,ds+
u' _{0}\int_{0}^{1}H _{u' _{0}} (su' ,su'',t,x)\,ds
$$
we obtain (see our comments regarding
\eqref{7.25.1})
$$
H (u' ,u'',t,x)-H (0,0,t,x)\leq \delta^{-1}\sum_{k}\lambda^{+}_{k}(u'')- \delta
\sum_{k}\lambda^{-}_{k}(u'')
$$
$$
+\delta^{-1}\sum_{i\geq1}|u' _{i}|
-\delta (u'_{0})^{+}+\delta^{-1} (u'_{0})^{-}=P _{0}(u' ,u'')
$$
$$
-\delta^{-1}\sum_{k}\lambda^{+}_{k}(u'') -(\delta
/2)\sum_{k}\lambda^{-}_{k}(u'') -\delta^{-1}\sum_{i\geq1}|u'_{k}|-
\delta^{-1} (u'_{0})^{-}-(\delta/2) (u'_{0})^{+}
$$ and \eqref{9.29.2}
follows since
$$
\big[\sum_{k}(\lambda^{+}_{k}(u'')
+\lambda^{-}_{k}(u'') ) \big]^{2}= \big(\sum_{k}|\lambda_{k}(u'')|
\big)^{2}
$$
$$
\geq\sum_{k}|\lambda_{k}(u'')|^{2}=
\sum_{i,j}|u''_{ij}|^{2}\geq d^{-2} \big(\sum_{i,j}|u''_{ij}|\big)^{2}.
$$

Estimate \eqref{3.6.3} follows from \eqref{9.29.2} and \eqref{3.6.1}.
Finally, the second assertion of the lemma follows directly from the
above Hadamard's formula.
 The lemma is proved.
\end{proof}

In addition, $H_{K}$ is boundedly inhomogeneous with respect to $u$ in
the sense that at all points of differentiability of $H_{K}(u,t,x)$ with
respect to $u$ \begin{equation}
                                              \label{9.30.01}
|H_{K}(u,t,x)-H_{Ku''_{ij}}(u,t,x)u''_{ij} -H_{Ku'_{r}}(u,t,x)u'_{r}|\leq
N(|H_{K}(0,t,x)| +K),
\end{equation} where $N$ depends only on $\delta$
and $d$.

Indeed, if \begin{equation}
                                            \label{9.30.2}
\kappa
 \big(\sum_{i,j}|u''_{ij}| +\sum_{i}|u'_{i}|\big)
 \geq  H^{+}(0,t,x)+K,
\end{equation} then by Lemma \ref{lemma 9.29.2}
$$
H (u,x)\leq P_{0}(u)-\kappa
 \big(\sum_{i,j}|u''_{ij}| +\sum_{i}|u'_{i}|\big)+
 H^{+}(0,t,x)\leq P_{0}(u)-K,
$$ so that $H_{K}(u,t,x)=P_{0}(u)-K$  and the left-hand side of
\eqref{9.30.01} is just $K$ owing to  the  fact that $P_{0}$ is positive
homogeneous of degree one. On the other hand, if the opposite inequality
holds in \eqref{9.30.2}, then  again in light of Lemma \ref{lemma
9.29.2}  the left-hand side of \eqref{9.30.01} is dominated by
$$
N\big(\sum_{i,j}|u''_{ij}| +\sum_{i}|u'_{i}|\big)+ |H_{K}(0,t,x)|\leq
N(|H_{K}(0,t,x)|+ H^{+}(0,t,x)+K),
$$
where
$$
H(0,t,x)\leq\max(H(0,t,x),-K)=H_{K}(0,t,x),$$
$$
H^{+}(0,t,x)\leq |H_{K}(0,t,x)|.
$$

Furthermore, as we have noticed above $H_{K}$ satisfies Assumption
\ref{assumption 9.23.1}  and \eqref{3.5.2} (with $\delta/2$ in place of
$\delta$) and as is easy to see $|H_{K}[0]|\leq |H[0]|+K$, which shows
that in the rest of the article we may (and will) assume that not only
Assumption \ref{assumption 9.23.1} and
 \eqref{3.5.2} are satisfied with $\delta/2$ in place of $\delta$
and \eqref{9.31.2} holds with a constant $N$, but also at all points of
differentiability of $H$ with respect to $u$ \begin{equation}
                                              \label{3.5.1}
|H (u,t,x)-H_{ u''_{ij}}(u,t,x)u''_{ij} -H_{ u'_{r}}(u,t,x)u'_{r}|\leq N_{0} ,
\end{equation} where $N_{0}$ is a constant and
 \begin{equation}
                                            \label{3.6.4}
H  \leq  P -\kappa
 \big(\sum_{i,j}|u''_{ij}| +\sum_{i}|u'_{i}|\big)+
|H (0,\cdot,\cdot)|,
\end{equation}
where $\kappa$ is the constant from Lemma \ref{lemma 9.29.2}.
By the way we keep track of the value of $\delta$ in Assumption \ref{assumption 9.23.1}  and
\eqref{3.5.2} because $P(u)$ is already fixed and defined by $d$ and
$\delta$.

As a result of the above arguments we see that to prove Theorem
\ref{theorem 9.23.1} it suffices to prove the following.
\begin{theorem}
                                   \label{theorem 10.5.1}
Suppose that Assumption \ref{assumption 9.23.1} is satisfied with
$\delta/2$ in place of $\delta$. Also assume that  \eqref{3.6.4} holds. Finally, assume that
estimate \eqref{9.31.2} holds for any $t,s\in \bR$, $x,y\in\bR^{d}$, and $u=(u',u'')$
with a constant $N$
 and \eqref{3.5.2} and \eqref{3.5.1}
hold  at all points of differentiability of  $H (u,t,x)$ with respect to
$u$.
Then the assertions of Theorem \ref{theorem 9.23.1}  hold true with $P$
introduced in the end of Section \ref{section 2.5.1}.
\end{theorem}

\begin{remark}
                                      \label{remark 8.4.1}
One may wonder why we need \eqref{3.5.1}
with a constant which does not enter the assertions
of Theorem \ref{theorem 10.5.1} in any way. The only reason
to reduce general $H$ to boundedly inhomogeneous ones
is that for those we can rewrite $H[v]$
 in such a way that only  pure
second-order derivatives of $v(t,x)$
with respect to $x$
enter. Then the whole operator $\max(H[v],P[v]-K)$
also has this form.

Another possible question is: Why don't we start
with $\max(H,P-K)$, which is already boundedly inhomogeneous
by the above? The point is that our way to transform
boundedly inhomogeneous operators does not preserve
the particular structure of $\max(H,P-K)$.
\end{remark}

\mysection{An auxiliary equation}

                                       \label{section 9.22.1}
Some notation in this section are different from the previous ones. Fix
an $h\in(0,1]$ and for $\xi\in\bR^{d}$ and any function $\phi$ on
$\bR^{d}$ introduce $$ T_{\xi}\phi(x)=\phi(x+h\xi),\quad
\delta_{\xi}=h^{-1}(T_{\xi}-1),\quad\Delta_{\xi}
=h^{-2}(T_{\xi}-2+T_{-\xi}). $$ Notice that $h$ enters the definition of
$T_{\xi}$ and $\delta_{\xi}$ and $\Delta_{\xi}$ are usual approximations
for the first and second-order derivatives along $\xi$.

Let $m\geq1$ be an integer and let
$\ell_{-m},...,\ell_{-1},\ell_{1},...,\ell_{m}$ be some fixed vectors in
$\bR^{d}$ such that $$ \ell_{-k}=-\ell_{k}. $$ Next denote
$\Lambda=\{\ell_{k}:k=\pm1,...,\pm m\}$, $$ \Lambda_{1}= \Lambda,\quad
\Lambda_{n+1}=
 \Lambda_{n}+ \Lambda ,\quad n\geq1,
\quad \Lambda_{\infty}=\bigcup_{n}\Lambda_{n}
\quad \Lambda^{h}_{\infty}=h\Lambda_{\infty}
\,. $$

Let $m'\geq0$ be an integer $\leq m$ and let $A=\{\alpha=(a,b,c)\}$ be a
closed bounded set in $\bR^{2m }\times\bR^{m'}\times \bR$, so that $$
a=(a_{-m},a_{-m+1},...,a_{-1},a_{1},...,a_{m})\in\bR^{2m}, $$ $$ b= (
b_{1},...,b_{m'})\in\bR^{ m'}, $$ and $c\in\bR$. Also let $f(\alpha,t,x)$
be a real-valued function defined for $\alpha\in A$,
$t\in\bR$, and $x\in\bR^{d}$.

Fix an $r\in\{1,...,m\}$ and for $k=\pm 1,...,\pm m$ set $$
\delta_{h,k}=\delta_{k}=\delta_{\ell_{k}},\quad
\Delta_{h,k}=\Delta_{k}=\Delta_{\ell_{k}}. $$

\begin{assumption}
                                   \label{assumption 9.18.1}
There are constants $\delta>0$ and $K_{1},K_{2} \in[0,\infty)$
  such that

(i) For any $(a,b,c)\in A$ and all $k$ we have $$ a_{k}=a_{-k},\quad
\delta \leq a_{k}\leq\delta^{-1},\quad|b_{k}|\leq\delta^{-1}, \quad
hb_{k}^{-}\leq  a_{k}, \quad c\geq0; $$

(ii) The function $f$ is continuous in $\alpha$ for any $(t,x)$ and
$|\delta_{r}f|\leq K_{1}$, $\Delta_{r}f\geq-K_{2}$ on $\bR^{d}$.

\end{assumption}

For $u=(u',u'')$ with
$$
u'=(u'_{0}, u'_{ 1},...,u'_{m'}),\quad
 u'' =(u''_{-m},...,u''_{-1},u''_{1},...,u''_{m}),
$$
introduce
$$
\cP(u,t,x)=\max_{\alpha=(a,b,c)\in A}\big(\sum_{|k|=1}^{m}
a_{k}u''_{k}+\sum_{ k =1}^{m'} b_{k}u'_{k} -cu'_{0}+f(\alpha,t,x)\big).
$$

For any function $u$ on $\bR^{d+1}$ define
$$
\cP[u](t,x)=\cP(u(t,x),\delta
u(t,x),\delta^{2}u(t,x) ,t,x),
$$ where $$ \delta u=( \delta_{1}u ,...,
\delta_{m'}u ), $$ $$ \delta^{2} u=(\Delta_{-m}u ,..., \Delta_{-1}u
,\Delta_{1}u ,..., \Delta_{m}u). $$ In connection with this notation a
natural question arises as to why use $\ell_{k}$ along with
$\ell_{-k}=-\ell_{k}$ since $\Delta_{k}=\Delta_{-k}$ and $$
a_{k}\Delta_{k}=2\sum_{k\geq1}a_{ k}\Delta_{k} $$ owing to the
assumption that $a_{k}=a_{-k}$. This is done for the sake of convenience
of computations. For instance,
$$
\Delta_{k}(uv)=u\Delta_{k}v+v\Delta_{k}u+(\delta_{k}u)(\delta_{k}v)
+(\delta_{-k}u)(\delta_{-k}v)
$$ (no summation in $k$). At
the same time
$$ a_{k}\Delta_{k}(uv)=ua_{k}\Delta_{k}v+va_{k}
\Delta_{k}u
+2a_{k}(\delta_{k}u)(\delta_{k}v)
$$
as if we were dealing with usual
partial derivatives.

Let
$Q^{o}$ be  a bounded subset
of $\bR\times \Lambda_{\infty}^{h}$, which is open
in the relative topology of
$\bR \times\Lambda_{\infty}^{h}$
and is
such that its projection on $\Lambda_{\infty}^{h}$ is
a finite set.
Introduce $\hat{Q}^{o}$ as the
set of points
$(t_{0},x_{0})\in\bR\times\Lambda_{\infty}^{h}$
for each of which there exists a sequence
$t_{n}\uparrow t_{0}$   such that
  $(t_{n},x_{0})\in Q^{o} $.
Observe that $Q^{o}\subset\hat{Q}^{o}$.
Also define
$$
Q=\hat{Q}^{o}\cup\{(t,x+h\Lambda):(t,x)\in Q^{o}\}.
$$

For $x\in  \Lambda_{\infty}^{h}$
we denote by $Q^{o}_{|x}$ the $x$-section of $Q^{o}$:
$\{t:(t,x)\in Q^{o}\}$.

In the future we will need the following.

\begin{lemma}
                                       \label{lemma 7.21.1}
Let $(a,b,c)(t,x)$ be
a bounded $\bR^{2m }\times\bR^{m'}\times \bR$-valued
(say $A$-valued) function on
$\bR^{d+1}$  satisfying $a_k\ge 0$,
$hb_{k}^{-}\leq  a_{k}$, and $c\ge 0$,
and let $v(t,x)$ be a bounded
function in $Q $ which is absolutely continuous
with respect to $t$
 on each open interval belonging to $Q^{o}_{|x}$
and for any $x\in  \Lambda_{\infty}^{ h}$
satisfies
$$
\partial_{t}v+Lv:=
\partial_{t}v+\sum_{|k|=1}^{m}a_{k}\Delta_{k}v
+\sum_{k=1}^{m'}b_{k} \delta_{k}v-cv=-\eta
$$
(a.e.) on each $Q^{o}_{|x}$, where $\eta=\eta(t,x)$
 is a bounded function. Redefine $v$ if necessary
for $(t,x)\in\hat{Q}^{o}\setminus Q^{o}$ so that
$$
v(t,x)=\nlimsup_{s\uparrow t,(s,x)\in Q^{o}}v(s,x).
$$
Finally, let $T$ be the width of
$Q^{o}$ in the $t$-direction.
Then in $Q^{o}$ we have
$$
v\leq T\sup_{Q^{o}}\eta_{+}+\sup_{Q\setminus Q^{o}}
v_{+}.
$$

\end{lemma}

\begin{proof}
Without losing generality we   assume that
$Q^{o}\in(0,T)\times\Lambda_{\infty}^{h}$. Then by considering
$$
v(t,x)-(T-t)[2\varepsilon+\sup_{Q^{o}}\eta_{+}],
$$
where $\varepsilon>0$, and then sending $\varepsilon\downarrow
0$,
we reduce the general case to the one with $\eta\leq-2
\varepsilon$.
Finally, we make one more harmless assumption that
$$
\sup_{Q}v>0.
$$

After that take a sequence $(t_{n},x_{n})\in Q$
such that
$$
v(t_{n},x_{n})\to \bar{v}:=\sup_{Q}v>0.
$$
If infinitely many points $(t_{n},x_{n})\not\in Q^{o}$,
then we have nothing to prove.

In the opposite case we may assume that $x_{n}=x_{0}$,
 $(t_{n},x_{n}) \in Q^{o}$ for all $n$, and the sequence
$t_{n}$ converges, say to $t_{0}$.
Denote by $I_{n}$ the  connected component (open interval)
of $Q^{o}_{|x_{0}}$ containing $t_{n}$.
By using subsequences
if needed and taking into account
the continuity of $v$ in $Q^{o}$
we come to  three  possibilities: either
$(t_{0},x_{0})\in\hat{Q}^{o}
\setminus Q^{o}$ and we have nothing to prove,
 or $(t_0,x_0)\in Q^o$,  or
else $t_{n}\downarrow
t_{0}$. Note that the second case can be
reduced to the third one by redefining
the $t_{n}$'s.

If the  third  possibility
 realizes, we claim that
\begin{equation}
                                           \label{7.22.1}
\nliminf_{n\to\infty}|I_{n}|=0 ,
\end{equation}
where $|I_{n}|$
is the length of $I_{n}$.

Indeed if \eqref{7.22.1} fails, then for all large $n$
the intervals $I_{n}$ coincide. Also in that case
there is an open interval $I\in\bR$ such that
$$
I\times\{x_{0}\}\subset Q^{o},\quad
I\times\{x_{0}+ h\Lambda\}\subset Q .
$$
Furthermore, $\partial_{t}v(t,x_{0})$ is bounded
on $I$, so that the limit of $v(t,x_{0})$
as $t\downarrow t_{0}$ exists and
$$
 \lim_{t\downarrow t_{0}}v(t,x_{0})=
\lim_{n\to\infty}v(t_{n},x_{0})=\bar{v}
$$
 In addition, $\bar{v} \geq v(t,x_{0})$
for $t>t_{0}$ and, since
$$
\partial_{t}v(t,x_{0})=
- Lv(t,x_{0})
-\eta(t,x_{0})
$$
for almost all $t\in I$, there exists
a sequence of
points $s_{n}\in I$ such that $s_{n}\downarrow t_{0}$ and
$$
Lv(s_{n},x_{0})+\eta(s_{n},x_{0})\geq-
\varepsilon
$$
implying that (recall that $\eta\leq-2\varepsilon$)
\begin{equation}
                                          \label{7.23.1}
Lv(s_{n},x_{0})\geq
\varepsilon.
\end{equation}

Next, consider the functions $v_{n}(x)=v(s_{n},x)$, $x\ne x_{0}$,
$v_{n}(x_{0})=\bar{v}$, for which $v_{n}(x_{0})\geq
v_{n}(x )$ for all $x\in x_{0}+h\Lambda$.
On the one hand, by the maximum principle we have
$Lv_{n}( s_{n},  x_{0})\leq0$ and, on the other
hand
$$ Lv_{n}(  s_{n}, x_{0})=Lv(s_{n},x_{0})+\xi_{n},
$$ where
$$
\xi_{n}=2h^{-2}\sum_{|k|=1}^{m}
a_{k}(s_{n},x_{0})[v(s_{n},x_{0})-\bar{v}]
$$
$$
+h^{-1}\sum_{k=1}^{m'}
b_{k}(s_{n},x_{0})[v(s_{n},x_{0})-\bar{v}]
+c(s_{n},x_{0})[v(s_{n},x_{0})-\bar{v}]
$$
and $\xi_{n}\to0$ as $n\to\infty$. This leads
to a contradiction with \eqref{7.23.1} and
proves \eqref{7.22.1}.

It follows that for infinitely many $n$, as $n$
increases, the value of $v$
at $(t_{n},x_{0})$
will become closer and closer
 to its value at the right end  points of
$I_{n}$'s since the time derivative of $v$ is bounded
and this proves the lemma.
\end{proof}

Next,
take a function $\eta\in C^{\infty}(\bR^{d})$ with bounded
derivatives, such that $|\eta|\leq1$ and set
$\zeta=\eta^{2}$, $$ |\eta'(x)| =|\eta'(x)|_{h} =\sup_{
k}|\delta_{ k}\eta  (x)|,\quad |\eta''(x)| =|\eta''(x)|_{h}
=\sup_{ k}|
\Delta_{ k}\eta  (x)|, $$ $$
\|\eta'\|=\|\eta'\|_{h}=\sup_{ \Lambda_{\infty}^{ h }}|\eta'
|_{h},\quad
\|\eta''\|=\|\eta''\|_{h}=\sup_{
\Lambda_{\infty}^{ h}}|\eta''
|_{h}, $$

Finally, let $u$ be a function on $\bR^{d+1}$ which is
continuously differentiable  with  respect
to $t$ and satisfies
\begin{equation}
                                                \label{9.19.9}
\partial_t u +\cP[u]=0 \quad\text{in}\,\,Q^{o}
\end{equation}
 and
\begin{equation}
                                        \label{eq4.48}
\partial_t u +\cP[u]\leq 0\quad
\text{on}\,\, Q \setminus Q^{o} . \end{equation}

\begin{theorem}
                                      \label{theorem 9.18.1}
There  exist constants $N=N(m,\delta)\geq1$
 and $N^{*}=N^{*}(m,\delta)$
such that for   any constant $\nu$ satisfying $$ \nu\geq N^{*}
\|\eta'\|+N (\|\eta''\|+\|\eta'\|^{2}), $$ we have in
$Q^{o}$ that  (recall
that  $a^{\pm}=(1/2)(|a|\pm a)$)
\begin{equation}
                                               \label{9.18.3}
\zeta^{2}[ ( \Delta_{r}u)^{-}]^{2}\leq \sup_{
 Q\setminus Q^{o} }\zeta^{2}[
( \Delta_{r}u)^{-}]^{2} +(N\nu+N^{*})\bar{W}_{r}+N\nu^{-2}K_{2}^{2} +
\nu^{-1}K_{1}^{2},
\end{equation}
where
$$
\bar{W}_{r}=\sup_{ Q }(|\delta_{r}u|^{2}+ |\delta_{-r}u|^{2}).
$$
Furthermore, $N^{*}=0$ if $b\equiv0$.

\end{theorem}

In the remaining part of this section no summation with respect to $r$ is performed. The number $r$
 is fixed at the beginning of the section. For simplicity of notation set $$
u_{rr}=\Delta_{r}u,\quad u_{r}=\delta_{r}u,\quad u_{kr} =-\delta_{-k}
\delta_{r}u. $$ Notice that in the above line the last notation when
$k=r$ is consistent with the first one.

Define
$$
u^{-}_{rr}= (u_{rr})^{-}.
$$
and for a constant $\nu\geq0$ introduce an operator
(recall that $r$ is fixed)
$$ L_{\nu}\phi=\zeta^{2}u_{rr}^{-}\Delta_{r}\phi -\nu\zeta
u_{r}\delta_{r}\phi. $$
Observe that
\begin{equation}
                                              \label{9.21.1}
L_{\nu}u=-\zeta^{2}(u_{rr}^{-})^{2}-\nu\zeta u_{r}^{2} =:-V_{\nu}.
\end{equation}
In the following lemma the fact that $u$ is a solution of
\eqref{9.19.9} is not used.

\begin{lemma}
                                        \label{lemma 9.20.1}
There exists $N=N(m,\delta)\geq1$ and $N^{*}=N^{*}(m,\delta)$ such that
if
\begin{equation}
                                              \label{12.23.4}
\nu\geq N^{*} \|\eta'\|+N (\|\eta''\|+\|\eta'\|^{2}) \end{equation} and
$N^{*}h\leq 1$, then  on $Q^{o}$ for any $\alpha=(a,b,c)\in A$ we have

$$ 2L_{\nu}[\partial_t+a_{k}\Delta_{k}+b_{k}\delta_{k}]u\geq
-[\partial_t+a_{k}\Delta_{k}+b_{k}\delta_{k}]V_{\nu} $$
\begin{equation}
                                            \label{9.20.3}
- (N\nu^{2}+N^{*}\nu)\bar{W}_{r}+(\nu/2)\zeta a_{k}u_{kr}^{2}.
\end{equation} Furthermore, $N^{*}=0$ if $b\equiv0$.
\end{lemma}

Up to an obvious formula $2L_{\nu}\partial_{t}
u=-\partial_{t}V_{\nu}$
this lemma is identical to Lemma 5.3 of \cite{Kr12.2}.

\begin{proof}[Proof of Theorem \ref{theorem 9.18.1}]
Denote by $N_{0}$ and
$N^{*}_{0}$ the constants $N$ and $N^{*}$ in Lemma \ref{lemma 9.20.1}
and take and fix a  $\nu$ satisfying \eqref{12.23.4} (with $N_{0}$ and
$N^{*}_{0}$ in place of $N$ and $N^{*}$).

Notice that in $Q^{o}$
$$
|u_{rr}|=h^{-1}|u_{r}+u_{-r}|\leq 2h^{-1}\bar{W}_{r}^{1/2}, $$ which
shows that \eqref{9.18.3} holds if $h\geq\nu^{-1/2}$ or if
$N^{*}_{0}h\geq1$. Therefore below we
  assume that
\begin{equation}
                                      \label{12.23.1}
h\leq\nu^{-1/2},\quad N^{*}_{0}h\leq1. \end{equation}

Let $(t_{0},x_{0})\in \bar{Q} $
be a point such that
$$
V_{\nu}(t_{0},x_{0})=\sup_{Q }V_{\nu}.
$$
If $(t_{0},x_{0})\not\in \bar{Q}^{o}$, then, as is easy to see,
this point can be approximated by points
lying in $Q\setminus Q^{o}$, in which case
$$
\sup_{Q }V_{\nu}=\sup_{Q\setminus Q^{o} }V_{\nu}
$$
and \eqref{9.18.3} follows. Therefore, in the rest of the proof
we may assume that
$$
(t_{0},x_{0})\in \bar{Q}^{o}.
$$

We  may also assume that
\begin{equation}
                                           \label{7.20.1}
\zeta(x_{0})u_{rr}^{-}(t_{0},x_{0})
>\nu h u_{ r}(t_{0},x_{0}).
\end{equation}
Indeed, if the opposite inequality holds, then
in light of \eqref{12.23.1}
in $Q^{o}$
$$
\zeta^{2}[u_{rr}^{-}]^{2}
 \leq V_{\nu}(t_{0},x_{0}) \leq
\nu^{2} h^{2} u^{2}_{ r}(t_{0},x_{0})
 +\nu \bar{W}_{r} \leq 2\nu \bar{W}_{r}.
$$

Next, consider two cases: 1) $\partial_{t}V_{\nu}(t_{0},x_{0})
>0$, 2) $\partial_{t}V_{\nu}(t_{0},x_{0})
\leq0$. In the first case $(t_{0},x_{0})\not \in Q^{o}$
and there is no sequence $t_{n}\downarrow t_{0}$
such that $(t_{n},x_{0}) \in Q^{o}$. Hence
$(t_{0},x_{0}) \in \hat{Q}^{o} \setminus Q^o $ and in $Q^{o}$
$$
\zeta^{2}[u_{rr}^{-}]^{2}
\leq V_{\nu}(t_{0},x_{0})\leq
\sup _{ \hat{Q}^{o}\setminus Q^{o} }\zeta^{2}[
( \Delta_{r}u)^{-}]^{2}
+\nu \bar{W}_{r} ,
$$
so that \eqref{9.18.3} holds.

In the remaining case
\begin{equation}
                                             \label{7.21.1}
\partial_{t}V_{\nu}(t_{0},x_{0})
\leq0.
\end{equation}
To extract some consequences of \eqref{7.21.1}, first we
notice that if a function $\phi(t_{0},x)$ is such that
 $\phi(t_{0},x)\leq\phi( t_0 ,x_{0})$
for $x\in x_{0}+h\Lambda$, then
owing to \eqref{7.20.1} at $(t_{0},x_{0})$ we have
$$
h^{2}L_{\nu}\phi(t_0,x_{0})=\zeta [ \phi(x_{0}+h\ell_{r})(\zeta
u_{rr}^{-} -\nu h u_{r}) +  \phi(x_{0}-h\ell_{r})\zeta u_{rr}^{-}]
$$
$$
-\zeta[2\zeta u_{rr}^{-}-\nu h u_{r} ]\phi \leq \zeta [ (\zeta
u_{rr}^{-} -\nu h u_{r})\phi + \zeta u_{rr}^{-} \phi ]-\zeta[2\zeta
u_{rr}^{-}-\nu h u_{r} ]\phi,
$$
where the last expression is zero. Thus
$$
L_{\nu}\phi(t_0,x_{0})\leq0,
$$
which in the terminology from \cite{Kr2}
means that $L_{\nu}$ respects the maximum principle.

Furthermore, we can find an  $\bar{\alpha} =(\bar{a},\bar{b},\bar{c})\in A$
such that
$$
\partial_t u(t_0,x_0)+\bar{a}_{k}\Delta_{ k}u(t_0,x_{0}) +\bar{b}_{k}\delta_{ k}u(t_0,x_{0})
-\bar{c}u(t_0,x_{0})+f(\bar{\alpha},t_0,x_{0})
$$
$$
= \partial_t u(t_0,x_0)+\cP[u](t_0,x_{0})=0.
$$
Since
$\partial_t u+\cP[u]\leq 0$ in $Q$,  we have that
$$
\phi(t_{0},x):= \partial_t u(t_{0},x)+\bar{a}_{k}\Delta_{
k}u(t_{0},x) +\bar{b}_{k}\delta_{ k}u(t_{0},x) -
\bar{c} u(t_{0},x)+f(\bar{\alpha},t_{0},x
)\leq 0
$$
for   $x\in x_{0}+h\Lambda$. Hence,
$0\geq
2L_{\nu}\phi(t_0,x_{0})$,
which owing to \eqref{9.21.1}, \eqref{9.20.3},
 and \eqref{7.21.1} yields
$$
0\leq[\partial_t+ \bar{a}_{k}\Delta_{k}+\bar{b}_{k}\delta_{k}- 2\bar{c}
]V_{\nu}(t_0,x_{0})- (\nu/2)\zeta \bar{a}_{k}
u_{kr}^{2}(t_0,x_{0})
$$
$$
+(N\nu^{2}+N^{*}\nu)\bar{W}_{r} - 2L_{\nu}f(\bar{\alpha},\cdot
)(t_0,x_{0} )
$$
$$
\leq[ \bar{a}_{k}\Delta_{k}+\bar{b}_{k}\delta_{k}- 2\bar{c}
]V_{\nu}(t_0,x_{0})- (\nu/2)\zeta \bar{a}_{k}u_{kr}^{2}(t_0,x_{0})
$$
$$
+(N\nu^{2}+N^{*}\nu)\bar{W}_{r} - 2L_{\nu}f(\bar{\alpha},\cdot
)(t_0,x_{0} ).
$$
Here the last term is dominated by
$$
 K_{2}\zeta^{2}u_{rr}^{-}(t_0,x_{0})+\nu |u_{r}(t_0,x_{0})| K_{1}
$$
$$ \leq  N\nu^{-1}K_{2}^{2}+(\nu/4)\zeta
\bar{a}_{k}u_{kr}^{2} (t_0,x_{0}) +
K_{1}^{2}+\nu^{2}\bar{W}_{r}.
$$
Furthermore, since $V_{\nu}(t_{0},x)\geq0$ attains its maximum at
$(t_{0},x_{0})$,
$$
[\bar{a}_{k}\Delta_{k}+\bar{b}_{k}\delta_{k}- 2\bar{c}
]V_{\nu}(t_0,x_{0})\leq0.
$$

We now conclude that
$$
(\nu/4)\zeta \bar{a}_{k}
u_{kr}^{2} (t_0,x_{0})\leq
(N\nu^{2}+N^{*}\nu)\bar{W}_{r} +N\nu^{-1}K_{2}^{2} + K_{1}^{2},
$$
which implies that in $Q^{o}$
$$
\zeta^{2}(u_{rr}^{-})^{2}\leq
V_{\nu}(t_0,x_{0})\leq N\zeta \bar{a}_{k}
u_{kr}^{2} (t_{0},x_{0})+\nu\bar{W}_{r}
$$
$$
\leq (N\nu+N^{*})\bar{W}_{r}+N\nu^{-2}K_{2}^{2} + \nu^{-1}K_{1}^{2}.
$$
Thus, estimate \eqref{9.18.3} holds on $Q^{o}$
in all cases and this
proves the theorem.
\end{proof}

\mysection{A model cut-off equation}

                             \label{section 10.18.2}

We will work in the setting of
 Section \ref{section 9.22.1}. However now  $h>0$ is not fixed.
Take a function $\cH(u,t,x)$, where $(t,x)\in\bR^{d+1}$,
$u=(u',u'')\in\bR^{1+m'+2m}$.

\begin{assumption}
                                   \label{assumption 9.23.01}
(i) The function $\cH$ is Lipschitz continuous in $u$ for every $(t,x)$, and
at all points of differentiability of $\cH$ with respect to $u$  we
have
$$ \delta\leq \cH_{u''_{k}}\leq\delta^{-1},\quad
  k=\pm1,...,\pm m,
\quad \delta\leq- \cH_{u' _{0}}\leq\delta^{-1}, $$ $$
|\cH_{u'_{k}}|\leq\delta^{-1},\quad k= 1,..., m'; $$

(ii) The number $\bar{\cH}=\sup_{t,x}|\cH(0,0,t,x)|$ is finite;

(iii) The function $\cH$ is locally Lipschitz continuous in
$(t,x)$ for every $u$ and there exists a constant $N'$ such that at all points of differentiability of $\cH$ with respect to $(t,x)$  we have $$
|\partial_{t}\cH (u,t,x)|+
|\cH_{x_{i}}(u,t,x)|\leq N'(1+|u|),\quad\forall\, i; $$

(iv) We have ${\rm Span}\,(\ell_{1},...,
\ell_{m})=\bR^{d}$.

\end{assumption}

Define
\begin{align*}
&\cP (u',u'',t,x)=\cP (u' ,u'')= 2\delta
^{-1}\sum_{k}(u''_{k})^{+} -(\delta /2)\sum_{k}(u''_{k})^{-}\nonumber\\
&\quad +2\delta ^{-1}\sum_{k\geq1}|u'_{k}| -(\delta/2) (u'_{0})^{+}+2\delta
^{-1} (u'_{0})^{-}\nonumber\\
&=\max_{\substack{\delta /2\leq a_{k}\leq 2/\delta \\|k|=1,...,m} }
\max_{\substack{ |b_{k}|\leq 2/\delta \\|k|=1,...,m'} } \max_{\delta/2
\leq c\leq 2/\delta }\big[\sum_{|i|=1}^{m} a_{i} u''_{i}+\sum_{ i
=1}^{m'}b_{i}u'_{i} -cu'_{0}\big].
\end{align*}

For  functions $v(t,x)$   introduce $$ H[v](t,x)=\cH(v(t,x),\partial
v(t,x),\partial^{2}v(t,x),t,x) $$ whenever and wherever it makes sense, where
$$
\partial v=(  v_{(\ell_{1})},...,v_{(\ell_{ m'})}),
$$
$$
\partial^{2} v=(
v_{(\ell_{-m})(\ell_{-m})},...,v_{(\ell_{-1})(\ell_{-1})},
v_{(\ell_{1})(\ell_{1})},...,v_{(\ell_{ m})(\ell_{ m})}), $$ and
$v_{(\ell)}=\ell_{i}v_{x_{i}}$, $v_{(\ell)
(\ell)}=\ell_{i}\ell_{j}v_{x_{i}x_{j}}$. Similarly,
$$
P[u](t,x)=\cP(u(t,x),\partial u(t,x),\partial^{2}u(t,x) ).
$$

Let $T\in(0,\infty)$, $\Omega$ be a bounded $ C^{2} $ domain in $\bR^{d}$, $g\in
C^{1,2}(\bar{\Omega}_T)$, and let $K\geq0$ be a finite number.
\begin{theorem}
                                       \label{theorem 9.14.1}
In addition to Assumption \ref{assumption 9.23.01} suppose
 that $\pm e_{i},\pm (e_{i}+ e_{j}),
  e_{i} -e_{j}\in \Lambda$, $i,j=1,..,d$,
 were  $e_{1},...,e_{d}$ is the standard orthonormal basis
in $\bR^{d}$  and  assume that all vectors in $\Lambda$ have
rational coordinates. Then there exists a unique $v\in
\cC^{1,1}(\bar{\Omega}_T)\cap
W^{1,2}_{\infty,\text{loc}}(\Omega_T)$ such that
$v=g$ on $\partial' \Omega_T$ and
\begin{equation}
                                              \label{9.14.2}
\partial_t v+ H _{K}[v]=0
\end{equation}
(a.e.) in $\Omega_T$, where
$$
 H _{K}[v]=
\max( H[v],P [v]-K).
$$
Furthermore,
\begin{equation}
                                             \label{1.14.3}
|v| ,  |D_{i}v| ,  \rho|D_{ij} v|,|\partial_t v|\leq N(\bar{\cH}+K+\|g\|_{C^{1,2}(\Omega_T)})
\end{equation}
in $\Omega_T$ (a.e.) for all $i,j$, where
 $N$ is a  constant depending only
on $\Omega$,   $\{\ell_{1},...,\ell_{m}\}$, $d$, and $\delta$ (but not on
$N'$).

\end{theorem}

To prove the theorem, we are going to use finite-difference
approximations of the operators $H[v]$ and $P[v]$.

For $h>0$ introduce
$$
P_{h}[v](t,x)=\cP (v(t,x),\delta_{h}v(t,x),
\delta_{h}^{2}v(t,x)),
$$
where
$$
\delta_{h} v=( \delta_{h,1}v ,..., \delta_{h,m'}v ), $$ $$
\delta^{2}_{h} v=(\Delta_{h,-m}v ,..., \Delta_{h,-1}v
 ,\Delta_{h,1}v
,..., \Delta_{h,m}v).
$$
Similarly we introduce
$$
H_{h}[v](t,x)
=H(v(t,x),\delta_{h}v(t,x),\delta_{h}^{2}v(t,x))
$$
and
$H_{K,h}[v]=\max(H_h[v],P_h[v]-K)$.

Here is Lemma 6.2 of \cite{Kr12.2}.
Its proof is similar to that  of Lemma \ref{lemma 9.29.2}.
\begin{lemma}
                                          \label{lemma 9.14.1}
Under Assumptions \ref{assumption 9.23.01} (i)  and
(ii),
\begin{equation*}
\cH \leq  \cP -(\delta /4)
 \big(\sum_{k}|u''_{k}| +\sum_{k}|u'_{k}|\big)+
\bar{\cH}.
\end{equation*}
\end{lemma}

Introduce $B$ as the smallest closed
 ball containing $\Lambda$ and set
$$ \Omega^{h}=\{x\in \Omega:x+hB\subset \Omega\} =\{x:\rho(x)\geq
\lambda h\}, $$ where $\lambda$ is the radius of $B$.

Consider the equation
\begin{equation}
                                              \label{2.25.3}
\partial_t v+H_{K,h}[v]=0\quad \text{in}\quad
[0,T]\times\Omega^{h}   \end{equation}
 with terminal-boundary condition
\begin{equation}
                                              \label{2.25.4}
v=g\quad\text{on}\quad \Big(\{T\}\times
\Omega^h\Big)\cup\Big((0,T)\times(\Omega\setminus
\Omega^{h})\Big) . \end{equation}

In view of Picard's method  of  successive  approximations, for any $h>0$ there
exists a unique bounded solution $v=v_{ h}$ of
\eqref{2.25.3}--\eqref{2.25.4}.
Furthermore, $\partial_{t}v_{h}$ is bounded and continuous.  By the way, we
do not include $K$ in the notation $v_{ h}$ since $K$ is a fixed
number.

Below by $h_{0}$ and $N$ with occasional indices we denote various
(finite) constants depending only on $\Omega$,
$\{\ell_{1},...,\ell_{m}\}$,
$d$, and $\delta$.

In the following   lemma  the additional assumption of Theorem
\ref{theorem 9.14.1} concerning the $e_{i}$'s is
not used.

\begin{lemma}
                                         \label{lemma 9.14.2}
Suppose that all vectors in $\Lambda$ have
rational coordinates and that
 Assumptions \ref{assumption 9.23.01} (i), (ii), and (iv)
are satisfied. Then there   are
constants  $h_{0}>0$  and $N$
 such that for all $h\in(0,h_{0}]$
and $|r|=1,...,m$
\begin{equation}
                                              \label{1.14.1}
|v_{ h}-g|\leq N(\bar{\cH}+K+\|g\|_{C^{1,2}(\Omega_T)})\rho,
\end{equation}
\begin{equation}
                                              \label{1.14.2}
  |\delta_{h,r}v_{ h}|
\leq N(\bar{\cH}+K+\|g\|_{C^{1,2}(\Omega_T)}),
\end{equation}
\begin{equation}
                                              \label{eq10.12}
  |\partial_t v_{ h}|
\leq N(\bar{\cH}+K+\|g\|_{C^{1,2}(\Omega_T)})
\end{equation}
on $\Omega_T$.
\end{lemma}
\begin{proof}
Introduce
$$
\cH_{K}=\max(\cH,\cP -K).
$$
As is easy to see, $\cH_{K}$
satisfies Assumption \ref{assumption 9.23.01}
 with $\delta/2$ in place of $\delta$. Therefore, by Hadamard's
formula
(cf. our comments about \eqref{7.25.1})
there exist functions $a_{k},b_{k}$, $k=\pm1,...,\pm m$, and $c$
such that
\begin{equation}
                                             \label{10.18.4}
\delta/2\leq a_{k}\leq2\delta^{-1},\quad|b_{k}|\leq2\delta^{-1}, \quad
\delta/2\leq c\leq2\delta^{-1}
\end{equation}
and in $(0,T)\times \Omega^{h}$ we
have
$$
-\cH_{K}[0]=\partial_t v_h+H_{K,h}[v_{h}]-\cH_{K}[0]=\partial_t v_h+
a_{k}\Delta_{h,k}v_{h}+b_{k}\delta_{h,k}v_{h}- cv_{h}
$$
$$
=\partial_t (v_h-g)+
a_{k}\Delta_{h,k}(v_{h}-g)+b_{k}\delta_{h,k}(v_{h}-g)- c(v_{h}-g)+f,
$$
where
$$
f= \partial_t g+a_{k}\Delta_{h,k}g+b_{k}\delta_{h,k}g- cg.
$$
After that
\eqref{1.14.1} is proved by using the barrier function $\Phi$ from
Lemma   8.8 of \cite{Kr11}  and the comparison principle (see, for instance, Section 5 of \cite{DK}).
In particular, \eqref{1.14.1} implies that
\begin{equation}
                                           \label{9.16.3}
|v_{ h}-g|\leq N_{1}(\bar{\cH}+K +\|g\|_{C^{1,2}(\Omega_T)})h \quad
\text{on}\quad (0,T)\times\big(\Omega\setminus\Omega^{3h}\big).
\end{equation}

Clearly, the remaining assertion of the lemma would
follow if we can prove
that \eqref{1.14.2} and \eqref{eq10.12}
 hold  on $\Omega_{T}
\cap[(0,T)\times(y+\Lambda_{\infty}^{h})]$ for
any
$y\in\bR^{d}$ with a constant $N$ independent of $h$  and $y$. Without
losing generality we concentrate on $y=0$ and observe that
the number of points in $\Omega^{2h}\cap
\Lambda_{\infty}^{h}$ is finite since the $\ell_{k}$'s
have rational coordinates.

To prove \eqref{1.14.2}, fix an $r$ and define
$$
Q^{o} =\{(t,x)\in
  (0,T)\times[\Omega^{2h}\cap
\Lambda_{\infty}^{h}]:(\delta/4)|\delta_{h,r}v_{h}|> \bar{\cH}+K\}. $$
If $Q^{o}=\emptyset$, then $(\delta/4)|\delta_{h,r}v_{h}| \leq
\bar{\cH}+K$ in $(0,T)\times[\Omega^{2h}\cap
\Lambda_{\infty}^{h}]$, and by virtue of \eqref{9.16.3},
$$
|\delta_{h,r}(v_{h}-g)| \leq 2N_{1}(\bar{\cH}+K
+\|g\|_{C^{1,2}(\Omega_T)})
$$
in $(0,T)\times\big(\Omega\setminus\Omega^{2h}\big)$. In that
case   \eqref{1.14.2} obviously holds.

Therefore, we assume that $Q^{o} \ne\emptyset$ and owing to Lemma
\ref{lemma 9.14.1} conclude that
\begin{equation}
                                              \label{9.16.7}
\partial_t v_h+P_{h} [v_{ h}]=K
\end{equation} in $Q^{o}$. Furthermore, \eqref{2.25.3}
implies that
\begin{equation}
                                              \label{9.22.3}
\partial_t v_h+P_{h}[v_{h}]\leq K
\end{equation}
in $(0,T)\times\Omega^{h}$.
Now use again the mean value theorem to conclude that
$$
\delta_{h,r}P_{h}[v_{h}]= a_{k}\Delta_{h,k}(\delta_{h,r}v_{h})
+b_{k}\delta_{h,k}(\delta_{h,r}v_{h})-c(\delta_{h,r}v_{h})
$$
for some functions $a_{k} $, $b_{k} $, and $c $ satisfying \eqref{10.18.4}.
In addition,
$$
\delta_{h,r}\big(\partial_t v_h+P_{h}[v_{h}]\big)\leq 0
$$
in $Q^{o}$ owing to
\eqref{9.16.7} and \eqref{9.22.3}, that is in
$Q^{o}$
$$
\partial_t \delta_{h,r}v_{h}+a_{k}\Delta_{h,k}(\delta_{h,r}v_{h})
+b_{k}\delta_{h,k}(\delta_{h,r}v_{h})-c(\delta_{h,r}v_{h}) \leq0.
$$

For small enough $h_{0}$  the operator $\partial_t+a_{k}\Delta_{h,k}
+b_{k}\delta_{h,k} -c$ with $h\in(0,h_{0}]$ respects the maximum
principle and therefore by Lemma \ref{lemma 7.21.1}
\begin{equation}
                                              \label{9.22.4}
\sup_{Q^{o}}(\delta_{h,r}v_{h})_{+}
\leq\sup_{ (0,T] \times[\Omega\cap\Lambda^{h}_{\infty}]
 \setminus Q^{o}}
(\delta_{h,r}v_{h})_{+}.
\end{equation}
While estimating the right-hand side of
\eqref{9.22.4}, notice that if $(t,x)\in
 (0,T] \times[\Omega\cap\Lambda^{h}_{\infty}]
 \setminus Q^{o}$, then one of the following
happens:

(i) $t=T$,

(ii) $t<T$ and $(t,x)\notin(0,T)\times\Omega^{2h}$,

(iii) $t<T$ and $(t,x)\in(0,T)\times\Omega^{2h}$ and
$(\delta/2)|\delta_{h,r}v_{h}| \leq \bar{\cH}+K$.

In case (i) we  have $v_h=g$,  in case (ii)
we may certainly use
\eqref{1.14.1},   and in case (iii) the estimate we need
is just given.

  It follows
that the right-hand side of \eqref{9.22.4} is dominated by
the right-hand side of
\eqref{1.14.2},
  if $h\in(0,h_{0}]$
and $h_{0}>0$  is sufficiently small.

Thus, in all situations
$$
(\delta_{h,r}v_{h})_{+}\leq N(\bar{\cH}+K
+\|g\|_{C^{1,2}(\Omega_T)})
$$
on $\Omega_T$. Upon replacing here $r$ with
$-r$, we get
$$
T_{h,-\ell_{r}}(\delta_{h,r}v_{h})_{-}\leq N(\bar{\cH}+K
+\|g\|_{C^{1,2}(\Omega_T)})
$$
in $(0,T)\times\Omega^{h}$, which after being combined
with the previous estimate proves \eqref{1.14.2}
in $\Omega^{h}$. In
$(0,T)\times\big(\Omega\setminus\Omega^{h}\big)$
estimate \eqref{1.14.2} has been established above.

Finally, we prove \eqref{eq10.12}.
This time denote
$$
Q^{o} =
\{(t,x)\in(0,T)\times[\Omega^{h}
\cap\Lambda^{h}_{\infty}]:
(\delta/4)\sum_k |\Delta_{h,k}v_{h}| > \bar{\cH}+K\}.
$$
Since $v_h$ satisfies \eqref{2.25.3}-\eqref{2.25.4},
estimate \eqref{eq10.12} obviously holds on $\Big(\{T\}\times \Omega^h\Big)\cup\Big((0,T)\times(\Omega\setminus \Omega^{h})\Big)$.
On $(0,T)\times[\Omega^{h}
\cap\Lambda^{h}_{\infty}]\setminus   Q^{o}$, we have
$$
(\delta/4)\sum_k |\Delta_{h,k}v_{h}| \le \bar{\cH}+K,
$$
which together with \eqref{1.14.1}, \eqref{1.14.2}, and
\eqref{2.25.3} implies that \eqref{eq10.12}
 holds on
$(0,T)\times[\Omega
\cap\Lambda^{h}_{\infty}]\setminus   Q^{o} $. Therefore, it
remains to establish  \eqref{eq10.12}
on $Q^{o}$ assuming that $Q^{o}\ne
\emptyset $.

Observe that equation
\eqref{9.16.7} holds on $Q^o$ by the same reasons as above.
 Every $x$-section of $Q^{o}$
is the union of open intervals  on which $\partial_{t}v_{h}$
is Lipschitz continuous   by virtue of \eqref{9.16.7}.
By subtracting the left-hand sides of
\eqref{9.16.7} evaluated at points $t$ and $t+\varepsilon$,
then transforming the difference by
using Hadamard's formula, and finally dividing by
$\varepsilon$ and letting $\varepsilon\to0$,
 we get that there exist functions $a_{k},b_{k},c$
satisfying \eqref{10.18.4} such that
on every $x$-section of $Q^{o}$ (a.e.)  we have
\begin{equation*}
\partial_t (\partial_t v_h) +[a_k\Delta_{h,k} +b_k
\delta_{h,k}-c](\partial_t v_h) =0.
\end{equation*}
 As above, owing to the continuity
of $\partial_{t}v_{h}$ with respect to $t\in[0,T]$
and Lemma \ref{lemma 7.21.1}, we conclude
$$
\sup_{Q^{o}}|\partial_t
v_h|\leq\sup_{ (0,T] \times[\Omega\cap\Lambda^{h}_{\infty}
] \setminus  Q^{o}}
|\partial_t v_h|,
$$
which implies \eqref{eq10.12} on $Q^{o}$. The lemma is proved.
\end{proof}

\begin{lemma}
                                          \label{lemma 9.16.1}
Suppose that Assumptions \ref{assumption 9.23.01} (i), (ii), (iv) are
satisfied. Assume also that all vectors in $\Lambda$ have rational
coordinates. Then there  are
constants $h_{0}>0$
and $N$ such that for
all $h\in(0,h_{0}]$ and $|r|=1,...,m$ \begin{equation}
                                              \label{9.16.6}
 (\rho-6\lambda h) |\Delta_{h,r}  v_{ h}|\leq N(\bar{\cH}+K
+\|g\|_{C^{1,2}(\Omega_T)}) \end{equation} on
$(0,T)\times \bR^{d}$ (we remind the reader that $\lambda$ is the radius of
$B$).
\end{lemma}
\begin{proof}
As in the proof of Lemma \ref{lemma 9.14.2} we will focus
on proving \eqref{9.16.6} in
$(0,T)\times
\Lambda_{\infty}^{h} $.  Then
for a fixed $r$ define
$$
Q^{o}:=\{(t,x)\in
(0,T)\times[\Omega^{3h}\cap
\Lambda_{\infty}^{h}] :(\delta/4)  |\Delta_{h,r}
v_{ h}(t,x)|> \bar{\cH}+ K \}.
$$ If $t\in (0,T)$, and $x\in  \Lambda^{h}_{\infty}$ is such
that $(t,x) \not\in Q^{o}$, then either $x\not\in \Omega^{3h}$, so that
$\rho(x)\leq 3\lambda h$ and \eqref{9.16.6} holds, or else $x \in
\Omega^{3h}$ but $(\delta/4)  |\Delta_{h,r} v_{ h}(t,x)|\leq  \bar{\cH}+
K$, in which case \eqref{9.16.6} holds again.

Thus we need only prove \eqref{9.16.6} on $Q^{o}$ assuming, of course,
that $Q^{o} \ne\emptyset$. By Lemma \ref{lemma 9.14.1} we have that \eqref{9.16.7} holds in
$Q^{o}$ and \eqref{9.22.3} holds in $Q\setminus Q^{o}$.

To proceed further observe
 a standard fact that there are constants $\mu_{0}>0$
and $N\in[0,\infty)$ depending only on $\Omega$  such that
 for any $\mu\in(0,\mu_{0}]$ there exists an $\eta_{\mu}
\in C^{\infty}_{0}(\Omega)$ satisfying
$$
\eta_{\mu}=1\quad\text{on}\quad \Omega^{2\mu},\quad
\eta_{\mu}=0\quad\text{outside}\quad \Omega^{\mu},
$$
\begin{equation}
                                                    \label{9.22.6}
|\eta_{\mu}|\leq1,\quad |D\eta_{\mu}|\leq
N/\mu,\quad|D^{2}\eta_{\mu}|\leq N/\mu^{2}.
\end{equation}
By Theorem \ref{theorem 9.18.1} and Lemma \ref{lemma 9.14.2}
 there are  constants $N$ and $h_{0}>0$
such that, for any number $\nu$ satisfying
$$
\nu\geq
N(\|\eta'_{\mu}\|_{h}+\|\eta'_{\mu}\|_{h}^{2}+ \|\eta''_{\mu}\|_{h}),
$$
we have in $Q^{o}$ that
$$
\eta_{\mu}^{4}[ (
\Delta_{r}v_{h})^{-}]^{2}\leq \sup_{  Q\setminus
Q^{o}}\eta_{\mu}^{4}[ (
\Delta_{r}v_{h})^{-}]^{2} + N(\nu+1) (\bar{\cH}+
K+\|g\|_{C^{1,2}(\Omega_T)})^{2}
$$
 if $h\in(0,h_{0}]$. We may certainly take $h_{0}$
smaller than $\mu_{0}/3$.
In light of \eqref{9.22.6}
one can take $\nu=N\mu^{-2}$ for an appropriate $N$ and then
$$
 \eta_{\mu}^{4}[ ( \Delta_{r}v_{h}(t,x))^{-}]^{2}\leq
\sup_{  Q\setminus Q^{o}}\eta_{\mu}^{4}[ ( \Delta_{r}v_{h})^{-}]^{2} + N \mu ^{-2}(\bar{\cH}
+ K+\|g\|_{C^{1,2}(\Omega_T)})^{2}
$$
for $(t,x)\in Q^{o}$.
While estimating the last supremum
we will only concentrate on $\mu\in[ 3h,\mu_{0}]$
($\ne\emptyset$), when
$\eta_{\mu}=0$ outside $\Omega^{3h}$. In that case, for any $(s,y)\in
  Q\setminus Q^{o}$, either $y\notin\Omega^{3h}$ implying that
$$
\eta_{\mu}^{4}[ ( \Delta_{h,r}v_{h})^{-}]^{2}(s,y)=0,
$$
or  $y \in\Omega^{3h}\cap \Lambda_{\infty}^{h}$ but
\begin{equation}
                                          \label{7.23.4}
(\delta/4)  |\Delta_{r}
v_{ h}(s,y)|\leq  \bar{\cH}+ K,
\end{equation}
or else ($(s,y)\notin Q^{o}$ and)
there is a sequence $s_{n}\uparrow s$ such that
$(s_{n},y)\in Q^{o}$.

The third possibility splits into two cases:
1) $s=T$, 2) $s<T$. In case 1 we have
$$
|\Delta_{r}
v_{ h}(s,y)|=|\Delta_{r}
g(s,y)|\leq N\|g\|_{C^{1,2}(\Omega_T)}.
$$
In case 2, owing to the continuity of $\Delta_{r}
v_{ h}(t,y)$ with respect to $t$, estimate
\eqref{7.23.4} holds again.

It follows that as long as $h\in(0,h_{0}]$, $(t,x)\in Q^{o}$, and  $\mu\in
[3h,\mu_{0}]$ we have
\begin{equation}
                                                   \label{9.22.7}
\eta_{\mu}^{4}[ ( \Delta_{r}v_{h})^{-}(t,x)]^{2}\leq
 N \mu ^{-2}(\bar{\cH}+ K+\|g\|_{C^{1,2}(\Omega_T)})^{2}.
\end{equation}

If $x$ is such that $\rho(x)\geq 6\lambda h$, take
$\mu=\mu_{0}\wedge(\rho(x)/(2\lambda))$, which is bigger than $3h$
provided that $h\leq\mu_{0}/3$. In that case also $ \rho(x)\geq 2\lambda \mu$, so that $\eta_{\mu}(x)=1$ and
we conclude from \eqref{9.22.7} that
$$ \rho(x)( \Delta_{r}v_{h})^{-}(t,x)
\leq N(\bar{\cH}+ K+\|g\|_{C^{1,2}(\Omega_T)}),
$$
\begin{equation}
                                                   \label{9.22.8}
(\rho(x)-6\lambda h)( \Delta_{r}v_{h})^{-}(t,x) \leq N(\bar{\cH}+
K+\|g\|_{C^{1,2}(\Omega_T)})
\end{equation}
for $(t,x)\in Q^{o}$ such that
$\rho(x)\geq 6\lambda h$. However, the second relation in \eqref{9.22.8}
is obvious for $\rho(x)\leq 6\lambda h$.

As a result of all the above arguments we see that \eqref{9.22.8} holds
in $(0,T)\times  \Lambda_{\infty}^{h}$ for any $r$ whenever $h\in(0,h_{0}]$.

Finally, since $\partial_t v_h+P_{h}[v_{h}]\leq K$ in $(0,T)\times \Omega^{h}$, we have that
$$
2\delta ^{-1}\sum_{r}(\Delta_{r}v_{h})_{+} \leq-\partial_t v_h+(\delta
/2)\sum_{r}(\Delta_{r}v_{h})_{-}
$$
$$ -2\delta
^{-1}\sum_{r\geq1}|\delta_{r}v_{h}| +(\delta/2) (v_{h})_{+}-2\delta
^{-1} (v_{h})_{-}+K,
$$
which after being multiplied by $\rho-6h$ along
with \eqref{9.22.8} and Lemma \ref{lemma 9.14.2} leads to \eqref{9.16.6}
on $(0,T)\times\Lambda_{\infty}^{h}$. As  is explained at the beginning of the proof,
this finishes proving the lemma.
\end{proof}

Mimicking the proof of Corollary 2.7 of \cite{Kr12.1}, we
obtain the following corollary from \eqref{1.14.2} and
\eqref{9.16.6}. Note that here Assumptions \ref{assumption
9.23.01}(iii) plays a crucial role  and only the
Lipschitz continuity in $x$ is needed.
\begin{corollary}
                                        \label{lem11.24}
Suppose that Assumption \ref{assumption 9.23.01} is
satisfied and all vectors in $\Lambda$ have rational
coordinates. Then there are constants $h_{0}>0$ and $M$, which may depend  on $N'$, such that for
all $h\in(0,h_{0}]$, $t \in (0,T]$, and $x,y \in\Omega$, we have
$$
|v_{h}(t,x)-v_{h}(t,y)|\leq
M(|x-y|+h).
$$
\end{corollary}

\begin{proof}[Proof of Theorem \ref{theorem 9.14.1}]
The theorem is proved in a the same way as Theorem 8.10
of \cite{Kr11} on the basis of Lemmas \ref{lemma 9.14.2} and
\ref{lemma 9.16.1} and the fact that the derivatives of $v$ are weak limits of
finite differences of $v_{h}$ as $h\downarrow0$.  Thanks to Lemma \ref{lemma 9.14.2}, for each $h$ sufficiently small and $x\in \Omega$,
$v_h(t,x)$ are uniformly bounded and equicontinuous in
$t\in [0,T]$. Let $Q$ be the subset of $\Omega$
consisting of points with rational coordinates. By
the Arzela--Ascoli theorem and Cantor's diagonal argument,  there is a sequence $h_n\to 0$ such that $v_{h_n}(t,x)$ converges
uniformly
on $[0,T]\times Q$. The limit  function $v(t,x)$ satisfies
\begin{equation}
                                    \label{eq12.11}
|v(t,x)-v(t,y)|\le M|x-y|
\end{equation}
for any $t\in [0,T]$ and $x,y\in Q$, where $M$ is from Corollary \ref{lem11.24}. Since $Q$ is dense in $\Omega$,
\eqref{eq12.11} allows us to extend $v$ to $\bar \Omega_T$,
with the extension denoted again by $v$ being continuous in $x$. Note that $v(t,x)$ is Lipschitz in $t$ with the Lipschitz constant bounded by the right-hand side of \eqref{eq10.12}, which is independent on $N'$. Moreover, by \eqref{1.14.1} $v = g$ on $\partial'\Omega_T$ and  $v_{h_n} (t,x)$ converges to $v(t,x)$ uniformly on $\Omega_T$.

Next we estimate the second term on the left-hand side of \eqref{1.14.3}. For any $\zeta\in C_0^\infty(\Omega_T)$ and for sufficiently small $h>0$, from \eqref{1.14.2} we have
$$
\left|\int_{Q_T}v_h\delta_{h,r}\zeta\,dx\,dt\right|
=\left|\int_{Q_T}\delta_{h,-r}v_h\zeta\,dx\,dt\right|\le  N(\bar{\cH}+K+\|g\|_{C^{1,2}(\Omega_T)})\max_{Q_T}|\zeta|
$$
for any $r=\pm 1,\ldots,\pm m$, where $N$ is independent of $h$. Passing to the limit as $h=h_n\to 0$, we obtain
$$
\sup_{Q_T}|v_{(\ell_r)}|\le N(\bar{\cH}+K+\|g\|_{C^{1,2}(\Omega_T)}).
$$
Similarly, using \eqref{9.16.6} we get
\begin{equation}
                                        \label{eq2.26}
\sup_{Q_T}|\rho
v_{(\ell_r)(\ell_r)}|
\le N(\bar{\cH}+K+\|g\|_{C^{1,2}(\Omega_T)}).
\end{equation}
Because $\pm e_{i},\pm (e_{i}+ e_{j}), e_{i} -e_{j}\in \Lambda$, $i,j=1,..,d$, using the identity
$$
2D_{ij}v=v_{(e_{i}+ e_{j})(e_{i}+ e_{j})}
-v_{(\ell_i)(\ell_i)}-v_{(\ell_j)(\ell_j)} ,
$$
we conclude from \eqref{eq2.26} that
\begin{equation*}
\sup_{Q_T}|\rho D_{ij}v|\le N(\bar{\cH}+K+\|g\|_{C^{1,2}(\Omega_T)}).
\end{equation*}
This completes the proof of \eqref{1.14.3}.

Finally, we show that $v$ is a unique solution of \eqref{9.14.2} with the terminal-boundary condition $v=g$ on $\partial'\Omega_T$. Since
$v\in W^{1,2}_{\infty,\text{loc}}(\Omega_T)$, at almost any point $(t_0,x_0)\in\Omega_{T}$ we have (see, for instance, Appendix 2
in \cite{Kr85})
$$
v(t_0+s,x_0+y)=P^{t_0,x_0}(s,y)+o(|s|+|y|^2),
$$
where
$$
P^{t_0,x_0}(s,y)=v(t_0,x_0)+y^i D_i v(t_0,x_0)+\frac 1 2 y^iy^j D_{ij}v(t_0,x_0)+s\partial_t v(t_0,x_0).
$$
Take $\varepsilon > 0$ and observe that for all small $r > 0$, $|o(8r^{2})|\le 3\varepsilon r^2$, which implies that for
$$
u(t,x) := P^{t_0,x_0} (t-t_0,x -x_0)+\varepsilon( t-t_0+| x - x_0|^2 -
 r^{2})
$$
we have
$$
u(t,x)\ge v(t,x) -| o(
8r^{2}) | +3\varepsilon r^2\ge v(
t,x)
$$
on $\partial' D$, where
$$
D = \{ (t,y) : t_0<t<t_0+4r^2,| y - x_0| < 2r \}.
$$ We modify $u$ outside $D$ so that it is smooth with bounded
derivatives in $\bR^{d+1}$. It then follows from the comparison
principle that for small enough $h$
\begin{equation}
                                        \label{eq3.22}
v_h(t_0,x_0)-u(t_0,x_0)\le \delta^{-1}
\sup_D(\partial_t u+H_{K,h}[u])_+
+\sup_{\partial_h D \cup \partial' D} (v_h-u)_+,
\end{equation}
where $\partial_h D=(t_{0},t_{0}+4r^{2})
\times\{y: 2r-\lambda h\leq|y-x_{0}|\leq 2r\}$.
Observe that $H_{K,h}[u]\to H_{K}[u]$ uniformly
in $D$ as $h\to 0$.
Taking $h=h_n\to 0$ in \eqref{eq3.22}, for sufficiently small $r>0$ we have
$$
\varepsilon r^2=v(t_0,x_0)-u(t_0,x_0)\le \delta^{-1}\sup_{D}(\partial_t u+H_{K}[u])_+.
$$
It follows that for any sufficiently small $r>0$, there is a point $(t_r,x_r)\in \bar D$ such that
\begin{equation}
                                        \label{eq3.45}
\partial_t u(t_r,x_r)+H_K[u](t_r,x_r)>0.
\end{equation}
Note that
\begin{align*}
\partial_t u(t_r,x_r)&=\partial_t v(t_0,x_0)+\varepsilon,\\
\partial_{\ell_k} u(t_r,x_r)&=\partial_{\ell_k} v(t_0,x_0)+O(r),\\
\partial^2_{\ell_k} u(t_r,x_r)&=\partial^2_{\ell_k} v(t_0,x_0)+O(\varepsilon).
\end{align*}
Letting $r\to0$ and then $\varepsilon \to0$ in \eqref{eq3.45}, we reach
$$
\partial_t v(t_0,x_0)+H_K[v](t_0,x_0)\ge 0.
$$
Similarly, we get an opposite inequality by considering
$$
u(t,x)= P^{t_0,x_0} (t-t_0,x -x_0)-\varepsilon( t-t_0+| x - x_0|^2 -
 r^2 ).
$$
Therefore, $v\in \cC^{1,1}(\bar{\Omega}_T)
\cap
W^{1,2}_{\infty,\text{loc}} (\Omega_T) $ is a solution to \eqref{9.14.2} with the terminal-boundary condition $v=g$ on $\partial'\Omega_T$. The uniqueness is a simple consequence of parabolic Alexandrov's estimate. The theorem is proved.
\end{proof}

\mysection{Proof of Theorem \protect\ref{theorem 10.5.1}}
                                      \label{section 12.13.4}

Here we suppose that the assumptions of Theorem \ref{theorem 10.5.1} are
satisfied and take the objects introduced in the end of Section
\ref{section 2.5.1}. Owing to the assumptions of Theorem
\ref{theorem 10.5.1} by Theorem   7.1 of \cite{Kr11} (see the beginning
of its proof in \cite{Kr11})
 there exists a function $\cH(z,t,x)$ defined for
$$
z  = ( z ',z ''),\quad z '= (z
'_{0},...,z '_{d}) \in\bR^{d+1},\quad z ''\in\bR^{m} ,\quad (t,x)\in\bR^{d+1}
$$ such that:

(i) The function $\cH $ is Lipschitz continuous in $z $ with Lipschitz
constant $\hat{\delta}^{-1}$ and there exists a constant $N'$ such that
$$
|\cH(z,t,x)-\cH(z,s,y)|\leq N'
(|t-s|+|x-y|)(1+|z|)
$$ for all $t,s\in \bR$, $x,y\in\bR^{d}$ and
$z$.

(ii) We have $\cH(z,t,x )=H(u,t,x)$ if $z '=u'$ and for all $j=1,...,m$
$$
z
''_{j}=\langle u''l_{j},l_{j}\rangle.
$$
In particular,
$\cH(0,t,x)=H(0,t,x)$ and if $v(t,x)$ is a real-valued function which is twice
differentiable at a point $x\in\bR^{d}$, at this point we have
$$
H[v](t,x)=\cH[v](t,x),
$$ where
$$
\cH[v](t,x)
=\cH(v(t,x),Dv(t,x)
,v_{( l_1)( l_1)}(t,x)
,...,v_{( l_m)( l_m)}v(t,x),t,x).
$$

(iii) At all points $(z,t,x) $ at which $\cH (z,t,x)$ is differentiable with
respect to $z$ we have \begin{equation}
                                                  \label{5.10.1}
|\cH _{z ' _{i}}(z,t,x)|\leq 4\delta ^{-1},\quad i=1,...,d,
\end{equation} \begin{equation}
                                                  \label{5.10.2}
 \delta/4\leq-\cH_{z'_{0}}(z,t,x )\leq 4\delta ^{-1},
\quad \hat{\delta}^{-1}\geq \cH_{z''_{j}}(z,t,x )\geq\hat{\delta},\quad
j=1,...,m . \end{equation}

The proofs in \cite{Kr11} use the fact that \eqref{3.5.1} holds and
yield the function $\cH$ such that, in addition, at all points $(z,t,x) $
at which $\cH (z,t,x)$ is differentiable with respect to $z$ we also have
$$
 |\cH(z,t,x )-\langle z , D_{z} \cH
 (z,t,x )\rangle |\leq 2N_{0}.
$$ However, the latter property of $\cH$ will not be used
 here,
so that we only used assumption \eqref{3.5.1} to be sure that $\cH$ with
the properties (i)-(iii) exists.

The functions $\cH$ from  above  and $\cP$ from
Section \ref{section 2.5.1} are instances of $\cH$ and $\cP$ from
Section \ref{section 10.18.2}. To see this, of course, one has to change
the constant $\delta$ in Section \ref{section 10.18.2}  and renumber the
$l_{i}$'s in Section \ref{section 2.5.1}. We also take into account that
$\hat{\delta}\leq\delta/4$ which allows us to match \eqref{5.10.1} and
\eqref{5.10.2} with the requirements of Assumption \ref{assumption
9.23.01} (i). Furthermore,  $\bar{\cH}=\bar{H}$. Therefore, Theorem
\ref{theorem 9.14.1} is applicable and
  yields a unique solution $v\in \cC^{1,1}(\bar{\Omega}_T)
\cap W^{1,2}_{\infty,\text{loc}} (\Omega_T) $ such that $v=g$ on $\partial' \Omega_T$,
estimates \eqref{1.14.3}, that is \eqref{1.13.1}, hold  true, and
$$
\partial v_t+\max[\cH(v,Dv,v_{(l_{1})(l_{1})},...,v_{(l_{m})(l_{m})},t,x),
$$
$$
\cP(v,Dv,v_{(l_{1})(l_{1})},...,v_{(l_{m})(l_{m})})-K]=0
$$
in $\Omega_{T}$ (a.s.). In light of the construction of
 $\cH$  this equation coincides with \eqref{9.23.2}, so
that
  the only remaining assertions of Theorem
\ref{theorem 10.5.1} to prove are that for $p>d+1$
\begin{equation}
                                                \label{10.19.1}
\|v\|_{W^{1,2}_{p}(\Omega_T)}\leq N_{p}(\bar{H}+K+\|g\|_{W^{1,2}_{p}(\Omega_T)})
\end{equation}
and estimate \eqref{2.28.1} holds. The latter follows
from other assertions of Theorem \ref{theorem 10.5.1} by Remark
\ref{remark 2.28.1}, so that we may concentrate on \eqref{10.19.1}.

Observe that
$$
\partial u_t+\max(H(u,t,x),P(u)-K)=
\partial u_t+ P(u)+G(u,t,x),
$$
where
$$
G(u,t,x)=(H(u,t,x)-P(u)+K)_{+}-K
$$ and, owing to condition \eqref{3.6.4},
$G(u,x)=-K$ if $$ \kappa\big(\sum_{i,j}|u_{ij}|+\sum_{i}|u_{i}|\big)
\geq\bar{H}+K. $$ If the opposite inequality holds, then
\begin{equation}
                                            \label{1.14.4}
|G(u,t,x)|\leq | H(u,t,x)-H(0,t,x)|+|P(u)|+\bar{H}+K \leq N(\bar{H}+K),
\end{equation} where $N$ depends only on $\delta$ and $d$. It follows
that the inequality between the extreme terms in \eqref{1.14.4} holds
for all $u$ and $(t,x)$. This allows us to apply Theorem 1.2 of
\cite{DKL} and shows that \eqref{10.19.1} holds if $v\in
W^{1,2}_{p}(\Omega_T)$. Since $P$ is convex with respect to
$u''$ and $G(v,t,x)$ is bounded, due to Theorem 1.2 of
\cite{DKL} there is a unique solution $w\in
W^{1,2}_p(\Omega_T)$ to the equation $ \partial
w_t+ P(w)=-G(v,t,x)$ with the terminal-boundary condition
$w=g$ on $\partial'\Omega_T$. By uniqueness of
$W^{1,2}_{d+1,\text{loc}} (\Omega_T)\cap
C(\overline{\Omega}_T)$-solutions we obtain
$w= v \in W^{1,2}_{p} (\Omega_T)$
and the theorem is proved.

\end{document}